\documentclass{amsart}
\usepackage{amsmath,amsthm,amsfonts,amssymb,amscd,overpic,mathrsfs}
\usepackage{dsfont}
\usepackage{amsthm}
\usepackage{amssymb}
\usepackage{graphicx}
\usepackage{amstext,amsthm,amsfonts,amsmath}
\usepackage[T1]{fontenc}
\usepackage{enumerate}
\usepackage{tikz}
\usepackage{pgfplots, pgfplotstable}
\usepackage[font=scriptsize]{caption}
\usepackage{indentfirst}
\usepackage{subfigure}
\newcommand{\Si}{{\Sigma}}

\newcommand{\eqdef}{\stackrel{\scriptscriptstyle\rm def}{=}}
\definecolor{Red}{cmyk}{0,1,1,0}

\definecolor{verde}{cmyk}{1,0,1,0}

\definecolor{azul}{cmyk}{1,1,0,0}

\begin{document}

\title[Attracting graphs of skew products]{Attracting graphs of skew products with non-contracting fiber maps}

\author[L. J.~D\'\i az]{Lorenzo J. D\'\i az}
\address{Departamento de Matem\'atica PUC-Rio, Marqu\^es de S\~ao Vicente 225,
Rio de Janeiro 22451-900, Brazil}
\email{lodiaz@mat.puc-rio.br}

\author[E. Matias]{Edgar Matias}
\address{Instituto de Matem\'atica Universidade Federal do Rio de Janeiro, Av. Athos da Silveira Ramos 149, Cidade Universit\'aria - Ilha do Fund\~ao, Rio de Janeiro 21945-909,  Brazil}
\email{edgarmatias9271@gmail.com}

\address{Departamento de Matem\'atica PUC-Rio, Marqu\^es de S\~ao Vicente 225,
Rio de Janeiro 22451-900, Brazil}

\begin{abstract}
We  study  attracting graphs of step skew products from the topological and ergodic points of view where 
the usual contracting-like assumptions of the fiber dynamics are replaced by  weaker merely topological conditions. In this context, we 
prove the existence of an attracting invariant graph and study its topological properties. We  prove the existence of globally  attracting measures and we show that (in some specific cases) the rate of convergence to these measures is
exponential.
\end{abstract}

\begin{thanks}{This paper is part of the PhD thesis of EM (PUC-Rio) supported by CAPES.
EM is currently supported by PNPD CAPES.
LJD is partially supported by CNPq and CNE-Faperj.
The authors warmly thank K. Gelfert for her useful comments on this paper.  
%This research has been supported  [in part] by Pronex,
%CNE-Faperj, and CNPq-grants (Brazil) and
%EU Marie-Curie IRSES ``Brazilian-European partnership in Dynamical
%Systems" (FP7-PEOPLE-2012-IRSES 318999 BREUDS). The authors 
%acknowledge the hospitality of their home institutions. LJD was partially supported by  Palis-Balzan project.  MR was supported by National Science Centre grant 2014/13/B/ST1/01033 (Poland)
}\end{thanks}

\date{}
\newtheorem{mteo}{Theorem}
\newtheorem{mcoro}{Corollary}
\newtheorem{mprop}{Proposition}
\newtheorem{mexample}{Example}

\newtheorem{teo}{Theorem}[section]
\newtheorem{prop}[teo]{Proposition}
\newtheorem{lema}[teo]{Lemma}
\newtheorem{schol}[teo]{Scholium}
\newtheorem{coro}[teo]{Corollary}
\newtheorem{defi}[teo]{Definition}
\newtheorem{example}[teo]{Example}
\newtheorem{remark}[teo]{Remark}
\newtheorem{nota}[teo]{Notation}
\newtheorem{claim}[teo]{Claim}
\newtheorem{fact}[teo]{Fact}

\numberwithin{equation}{section}

\maketitle

\section{Introduction}
\label{ss.introstep}

We  study  attracting graphs of step skew products from the topological and ergodic points of view. 
Step skew products naturally arise in the study of  random maps, but they are also an important tool in the construction of topological dynamical systems with a very rich dynamics.  In this paper, we combine both aspects and give criteria for the existence of global topological attractors and study their ergodic properties. As this is a very wide field which typically combines various approaches, we will refer to some key references only in the course of this introduction.

 Given 
 a  compact metric space $M$,
  finitely many continuous maps  $f_{i}\colon M\to M$, $i=1,\dots,k$,
  and the
 shift map $\sigma$ on the space of sequences $\Sigma_{k}=\{1,\dots, k\}^{\mathbb{Z}}$, the map
$F\colon \Sigma_{k}\times M\to  \Sigma_{k}\times M$
 defined by 
\begin{equation}\label{defiskew2} 
F(\vartheta,x)=(\sigma(\vartheta),f_{\vartheta_{0}}(x))
\end{equation}
is called a \emph{step skew product} over the shift  $\sigma$, where $\vartheta =(\vartheta_i)_{i\in \mathbb{Z}}$. The maps $f_{i}$ are called the \emph{fiber maps} of $F$ 
and $\sigma$ is the \emph{base map}. 

%
%A central object in this study is the so-called 
%{\emph{maximal attractor}} of $F$ defined by 
%\begin{equation}
%\label{e.maximal}
%\Lambda_{F} \eqdef \bigcap_{n\geq 0}
%F^{n}(\Sigma_k \times M).
%\end{equation}
% It turns out that
% \begin{equation}
% \label{e.defspine}
% \Lambda_F= \bigcup_{\vartheta\in \Sigma_k} \{\vartheta\} \times I_\vartheta,
% \quad
% \mbox{where}
% \quad
% I_{\vartheta}\eqdef \bigcap_{n\geq 0} f_{\vartheta_{-1}}\circ\cdots\circ f_{\vartheta_{-n}}(M),
% \end{equation}
% where the set $I_\vartheta$ is called the {\emph{spine}} of $\vartheta$,
%see Proposition~\ref{p.maximalattractor}.
We consider the set of \emph{weakly hyperbolic sequences}
\begin{equation}\label{e.stmenos}
S_{F}\eqdef\big\{\vartheta\in \Sigma_{k}\colon 
\mbox{$\bigcap_{n\geq 0} f_{\vartheta_{-1}}\circ\cdots\circ f_{\vartheta_{-n}}(M)$   is a singleton}\big\} 
\end{equation}
and the \emph{coding map} defined by 
\begin{equation}
 \label{e.varrhoproj}
\rho\colon S_{F}\rightarrow M, \quad 
\rho(\vartheta)\eqdef\lim_{n\to \infty}f_{\vartheta_{-1}}\circ\cdots\circ f_{\vartheta_{-n}}(p).
\end{equation}
By the definition of the set  $S_{F}$, the limit above always exists and does not depend on the choice of
the point $p\in M$. 
The image of $S_F$ by $\rho$ is denoted by $A_{F}$ and called the \emph{target set}.

Note that for every $\vartheta \in S_F$ it holds
\begin{equation}\label{e.conjugation}
 f_{\vartheta_0} (\rho (\vartheta) )= \rho ( \sigma (\vartheta)) 
\end{equation}
and therefore the graph of $\rho$, (${\mathrm{graph}\,\rho}$), is $F$-invariant.
Invariant graphs are well studied objects when $\sigma$ is a measure-preserving transformation of
some  probability
space $(\Sigma_k, \mathscr{F}, \mathbb{P})$ such that
equation \eqref{e.conjugation} holds $\mathbb{P}$-a.e.. The study of these invariant graphs is well
understood in the case when the fiber maps satisfy some form of hyperbolicity (contraction), as for instance, 
having a negative Lyapunov exponent (i.e., non-uniformly contracting). 
In this setting and assuming that the fiber maps are Lipschitz, a classical result by Stark
\cite{Stark} claims that there is a unique invariant graph and that it is an attractor. Note that in this setting 
we have  $\mathbb{P} (S_F)=1$.
We observe that
there is another line of research, of some different flavor, concerning graphs of skew products with minimal base dynamics, see for example \cite{Jager} and the references therein. 

In this paper, we analyze invariant graphs in much more general contexts. First, we only assume  that the fiber maps are continuous. Second, we do not assume any contracting-like hypotheses.
%For instance, our study applies to some systems having "expanding'' periodic points.
 The only (purely topological) requirement 
is that the set $S_F$ of weakly hyperbolic sequences is nonempty. For systems satisfying this condition, we state simple topological properties that
imply that the invariant graph is the attractor of the system and is transitive.
If, in addition, we assume that $S_F$ has probability one (note that this does not involve any contracting-like property) we obtain also ergodic information about the invariant graph, which is attracting from the ergodic point of view. 
In this context, 
invoking a variation of the 
``splitting property''  of Dubins and Freedman \cite{Dubins} introduced in \cite{DiazMatias},
 we put Stark's result \cite{Stark} into a much more general framework.
 In some sense, this topological-like ``splitting property''  replaces in a topological way Stark's hypotheses on the Lyapunov exponent. 

We now briefly describe our results. The complete statements and definitions are given in Section~\ref{ss.statements}.
Let us 
start with the topological results. The
{\emph{maximal attractor}} $\Lambda_F$ of $F$ is defined by
\begin{equation}
\label{e.maximal}
\Lambda_{F} \eqdef \bigcap_{n\geq 0}
F^{n}(\Sigma_k \times M).
\end{equation}
First note that  $\overline{\mathrm{graph}\,\rho} \subset \Lambda_F$ and that, in general, this inclusion can be strict.
Observe that the $\omega$-limit set for $F$ of any point is contained in the maximal attractor.
We also observe that, in general, the union of the $\omega$-limit sets  of  ``generic points''
may be properly contained in the maximal attractor. Hence, when considering  ``attracting sets'',
the maximal attractor may be ``excessively big''. We will see that the set 
$\overline{\mathrm{graph}\,\rho}$ turns out to be ``the actual attractor'' of the system.

Theorem~\ref{mt.mixing} states that the closure of ${\mathrm{graph}\,\rho}$ is a topologically mixing 
$F$-invariant set. Moreover, there is a ``large''  subset of $\Sigma_k\times M$
consisting of points whose $\omega$-limit contains the set $\mathrm{graph}\, \rho$.
Further, we state under which conditions the  closure of 
${\mathrm{graph}\,\rho}$
and the $\omega$-limits coincide and when the closure of the graph of $\rho$ coincides with the maximal attractor.

Our second main result is an application of Theorem~\ref{mt.mixing}. It
is motivated by  \cite{Volk}  and deals with skew products with one dimensional fiber maps (i.e., 
$M=[0,1]$). 
Theorem~\ref{mt.espinhodesconexo} states that the structure of  $\overline{\mathrm{graph}\,\rho}$ can be much more complicated than perhaps expected.
Note that if  $M=[0,1]$ then $\Lambda_F \cap ( \{\vartheta\} \times [0,1])$ is either a point or an interval.
 In \cite{Volk} it is shown that generically, in the class of step skew products whose
 fiber maps are $C^{1}$ increasing diffeomorphisms of $[0,1]$ with image strictly inside $[0,1]$, there 
 is finitely many ``locally maximal attractor'' that is a {\emph{bony attractor,}} i.e., a closed set that
intersects almost every (with respect to some Markov measure) fiber at a single point and any other fiber at an
interval (a ``bone''). However, exactly as in the discussion above about maximal attractors, these locally maximal attractors may be excessively big in order to capture appropriately the forward asymptotic behavior. 
Hence, in what follows, we focus on the topological structure of $\overline{\mathrm{graph}\,\rho}$.
Theorem~\ref{mt.espinhodesconexo} states that for each $m\ge 1$ there are  robust examples for which there is a dense subset $\Gamma$ in the symbolic space $\Sigma_k$ such that for each $\vartheta\in \Gamma$ the intersection of the fibre $\{\vartheta\}\times [0,1]$ with 
 $\overline{\mathrm{graph}\,\rho}$ is the
 union of $m$ disjoint intervals. 

Continuing with the ergodic results, we
consider the shift map $\sigma$  as a measure-preserving transformation 
of some probability space
$(\Sigma_{k},\mathscr{F}, \mathbb{P})$. Note that if $\mathbb{P}$ is ergodic then 
 $\mathbb{P}(S_{F})=0$ or $1$.
Under the assumption $\mathbb{P}(S_{F})=1$, we give a complete description
  of the dynamics of the push forward $F_*$ of $F$ (i.e., $F_* \nu (A)= \nu (F^{-1}(A)$)
  in the space of  
  probability
  measures $\mu$ with marginal $\mathbb{P}$,
  \begin{equation} \label{e.marginal}
   \mathcal{M}_\mathbb{P}  (\Sigma_{k}\times M) \eqdef \{ \mu \in \mathcal{M}_1 (\Sigma_{k}\times M) \colon
  \Pi_\ast \mu =\mathbb{P}\}, 
  \end{equation}
  where 
  $ \mathcal{M}_1(\Sigma_{k}\times M)$ is the space of probability measures of $\Sigma_k\times M$ and
  $\Pi\colon \Sigma_k \times M \to \Sigma_k$ is the 
  projection, $\Pi(\vartheta, x) =\vartheta$.
  Note that the set $\mathcal{M}_\mathbb{P}(\Sigma_{k}\times M)$ is compact and $F_\ast$-invariant (see 
Proposition~\ref{p.marginal} for details).
  
  Our third main result, Theorem~\ref{mt.skewproduct}, states that
  there is a measure $v_\mathbb{P} \in  \mathcal{M}_\mathbb{P} (\Sigma_{k}\times M)$ 
  supported on 
  the  graph of $\rho$ 
such that   (in the weak$\ast$ topology)
  \begin{equation}\label{letacc}
 \lim_{n\to \infty}F_{*}^{n}\nu=v_\mathbb{P}, \quad \mbox{for every $\nu\in  \mathcal{M}_\mathbb{P} (\Sigma_{k}\times M) $.}
 \end{equation}
 Therefore, we call the measure $v_\mathbb{P}$ the {\emph{global attractor}} of 
$F_{*}|_{\mathcal{M}_{\mathbb{P}}}$. Note that this result is a version of the Letac's contraction principle for skew products, \cite{Letac}. We emphasize that the probability $\mathbb{P}$ does not have to be  Bernoulli. 
Finally, let us observe that the idea of 
 using the ``reversed compositions'' (in the spirit of  \eqref{e.varrhoproj} and explicitly stated in the definition of the
 set $S_F$) to get ergodic properties can be already found in 
 the pioneering works by Furstenberg, see \cite{Fur}. Later, Letac \cite{Letac} used this method
 to prove the existence and uniqueness of the stationary measure in the context i.i.d. random products.
 
%
%Note that $\mathbb{P}(S_{F})=1$ implies that $f_{\omega_{0}}(\rho(\omega))=\rho(\sigma(\omega))$ for 
%$\mathbb{P}$-almost every $\omega$. We observe that in the study of skew products the graphs of maps
% satisfying this type of ``invariance equation'' are called  \emph{invariant maps}.

 Our last main result, Theorem~\ref{exponentialdw}, is a strong version of Theorem~\ref{mt.skewproduct} 
in the case when $M$ is a compact subset of $\mathbb{R}^m$ and
$\mathbb{P}$ is a Markov measure and stated for 
step skew products satisfying the \emph{splitting
property} introduced  in
\cite{DiazMatias} (see Definition \ref{splitgene2}). This property is
 a variation 
of the classical splitting property introduced by
Dubins and Freedman
in \cite{Dubins} to study 
Markov operators associated to (i.i.d.) random products.
In \cite{DiazMatias} we prove that this  property implies  $\mathbb{P}(S_F)=1$ and hence,
by Theorem~\ref{mt.skewproduct}, there is a global attracting measure  $v_\mathbb{P} \in \mathcal{M}_\mathbb{P}$. Theorem~\ref{exponentialdw} states that rate of convergence to the attracting measure is exponential
(with respect to the Wasserstein metric).

We conclude this introduction observing that our results were motivated by the
constructions in \cite{Yuri} and \cite{DiGe} of bony attractors and porcupine-like horseshoes (see the discussion in Section~\ref{ss.examples}) but that our results can be applied to a much more wider contexts, this is illustrated with a series of examples in 
Section~\ref{ss.statements}.
The goal is  to understand the dynamics of these examples
from the topological and ergodic points of view.

\subsection*{Organization of the paper}
In Section~\ref{ss.statements} we state precisely the main results in this paper and also provides some classes of examples. Each subsequent section is dedicated to the proof of one of the theorems above. Thus Theorem $k$ is proved in Section $k+2$, where $k=1, \dots, 4$.

\section{Statement of results}\label{ss.statements}

Let $M$ be a compact  metric space and $F\colon \Sigma_{k}\times M\to  \Sigma_{k}\times M$ a step skew product as in~\eqref{defiskew2} 
 over the shift map with fiber maps $f_{1},\dots,f_{k}\colon M\to M$. 
 We use the following notation for compositions of fiber maps,
\begin{equation}\label{eq:defcompo}
	f_\vartheta^n\eqdef f_{\vartheta_{n-1}} \circ \cdots \circ f_{\vartheta_{0}},
	\quad\mbox{ where }
	\vartheta\in \Sigma_{k}.
\end{equation}	
We denote by $\mathrm{IFS}(f_{1},\dots,f_{k})$ the \emph{iterated function system} induced 
by the maps $f_{i}$ (i.e., the set of maps $g$ of the form $g=f_{\omega_0} \circ \cdots \circ f_{\omega_i}$, 
where $\omega_j\in \{1,\dots,k\}$ and $i\ge 0$).

\subsection{Attracting invariant graphs}
 A key object in our study is the
 \emph{Barnsley-Hutchinson operator} induced by $F$  that associates to each set $A\subset M$ the set
\begin{equation}\label{e.BH}
\mathcal{B}_{F}(A)\eqdef  \bigcup_{i=1}^{k} f_{i}(A).
\end{equation}
This operator acts continuously in the space of nonempty compact subsets 
of $M$ endowed with the Hausdorff 
metric.
A fixed point $A$ of the Barnsley-Hutchinson operator
is  \emph{Lyapunov stable} if for every open neighborhood $V$ of $A$ there
is an open  neighborhood $V_{0}$ of $A$ such that  
%\marginpar{$\mathcal{B} (\overline{A_{\mathrm{t}}})= \overline{A_{\mathrm{t}}}$???}
\begin{equation}
\label{e.stableBH}
\mathcal{B}_{F}^{n}(V_{0})\subset V \quad \mbox{for every}\quad n\geq 0.
\end{equation}

Given a compact $F$-invariant subset $Y$
 of $\Sigma_k\times M$, the map $F$ is
\emph{topologically mixing}
in $Y$ if  for every pair of nonempty relatively open sets $U$ and $V$ of $Y$ 
 there is $n_0$ such that
$F^n(V)\cap U\ne \emptyset$ for all $n\ge n_0$.
Given $z=(\vartheta,x)$
we denote by $\omega_{F} (z)$ the $\omega$-limit set of $z$ for $F$ (i.e., the points $(\eta,y)$ such that there is a sequence $n_k \to \infty$ such that
$F^{n_k}(z)\to (\eta,y)$).
We say that a sequence $\vartheta\in \Sigma_k$ is \emph{disjunctive} if its forward 
$\sigma$-orbit  is dense in $\Sigma_k$.
A point $z=(\vartheta,x)$ such that $\vartheta$ is disjunctive is called 
a \emph{disjunctive point}.

Recall the definition of the target set $A_F =\rho (S_F)$.
Denote by $F_{A}$ the restriction of $F$ to the set $\Sigma_{k}\times \overline{A_{F}}$
and by $\Lambda_{F_{A}}$
the maximal attractor of $F_{A}$, that is,
\begin{equation}\label{e.maximalattractor}
\Lambda_{F_A} \eqdef \bigcap_{n\geq 0}
F_{A}^{n}(\Sigma_k \times \overline{A_{F}}).
\end{equation}
%bal

\begin{mteo}[Topological attractor]\label{mt.mixing} 
Let $M$ be a compact  metric space and $F\colon \Sigma_{k}\times M\to  \Sigma_{k}\times M$ a step skew product as in~\eqref{defiskew2}. 
Assume that $S_{F}\neq \emptyset$. Then:
\begin{itemize}
\item[(1)]
For every disjunctive point $z$  it holds 
$\overline{\mathrm{graph}\,\rho}\subset\omega_F(z)$. Moreover, assume that
\begin{itemize}
\item either $\overline{A_{F}}$ is a Lyapunov stable fixed point
of the Barnsley-Hutchinson operator $\mathcal{B}_{F}$
\item 
or  $\overline{A_{F}}$
 has nonempty interior, 
\end{itemize} 
 then
$
\overline{\mathrm{graph}\,\rho}=\omega_F(z).
$
%\item
%For every $z=(\xi,x)$ such that $\xi$ stat-disjunctive it holds 
%$\overline{\mathrm{graph}\,\varrho}\subset\omega_{stat}(z)$
\item[(2)]
$
\Lambda_{F_{A}}=\overline{\mathrm{graph}\,\rho}.
$
\item[(3)]
%$F(\overline{\mathrm{graph}\,\varrho})=\overline{\mathrm{graph}\,\varrho}$ and
$F$ is topologically mixing
in $\overline{\mathrm{graph}\,\rho}$.
%\item
%If $X$ is  an $m$-locally Euclidean compact metric space 
%then
%$$
%\overline{\mathrm{graph}\, \varrho}\subset \overline{\mathrm{per}(F)}\subset \Lambda_F.
%$$
\end{itemize}
\end{mteo}

\begin{remark}
{\em{
Note that if $\overline{A_{F}}=M$ then item (2) implies $
\Lambda_{F}=\Lambda_{F_A}=\overline{\mathrm{graph}\,\rho}.$
A particular  setting where  $\Lambda_F=\overline{\mathrm{graph}\,\rho}$ 
is given by systems whose phase space 
is the {\emph{Hutchinson attractor}}\footnote{The Hutchinson attractor is the globally attracting fixed of the
Barnsley-Hutchinson operator
$\mathcal{B}_{F}$
when the maps $f_i$ are uniform contractions, see \cite{Hu}}
of a uniformly contracting subsystem, that is, 
there are uniform contractions
$g_1,\dots,g_\ell \in \mathrm{IFS}(f_{1},\dots,f_{k})$  such that the union of the sets $g_i(M)$, $i=1,\dots,\ell$, covers $M$. Note that these conditions imply that $A_{F}=M$ is the Hutchinson attractor, see \cite{Hu}.
These ideas were used in \cite{Yuri} to get robust examples of skew products whose maximal attractors
are graphs.}}
\end{remark}

\begin{example}\label{e.AFinterior}
{\em{Let  $K\subset M$ be a compact subset of $M$ with  nonempty interior and 
$f_1,\dots, f_k\colon M\to M$ uniformly contracting maps such that
\begin{equation}\label{e.k=k}
\bigcup_{i=1}^k f_i(K)=K.
\end{equation}
Then for every continuous maps $f_{k+1},\dots , f_{k+\ell} \colon M\to M$, $\ell \ge 0$, the step skew product
$F\colon \Sigma_{k+\ell} \times M\to \Sigma_{k+\ell} \times M$  with fiber maps $f_1,\dots, f_{k+\ell}$ 
is such that $K\subset \overline{A_F}$.
Thus $ \overline{A_F}$ has nonempty interior. Therefore we can 
apply Theorem~\ref{mt.mixing} to $F$.  Note that if $\ell \ge 1$ then the set $\overline{A_F}$ is, in general, unknown, but still we can conclude the existence of an invariant attracting graph.

\begin{figure}[h!]
\centering
\begin{tikzpicture}[xscale=3,yscale=3]
 \node[scale=0.6,below] at (0.5,0) {$K$};
   \node[scale=0.6,below] at (0.9,0.67) {$f_{1}$};
    \node[scale=0.6,below] at (0.3,0.9) {$f_{2}$};
      \node[scale=0.6,below] at (0.15,0.67) {$f_{3}$};

\draw[->] (0,0)-- (1,0);
\draw[->] (0,0)--(0,1);
%\draw[-]  (0.4,0)--(0.7,0);
\draw[-](0.4,0.02)--(0.4,-0.02);
\draw[-](0.7,0.02)--(0.7,-0.02);

%\draw[red,-] (0,0.4)--(1,0.9);
\draw[dashed] (0,1)--(1,1);
\draw[dashed] (0,0)--(1,1);
\draw[dashed] (1,0)--(1,1);
\draw[dashed] (0.7,0.7)--(0.7,0.4);
\draw[dashed] (0.4,0.7)--(0.7,0.7);
\draw[dashed] (0.4,0.4)--(0.4,0.7);
\draw[dashed] (0.4,0.4)--(0.7,0.4);
%\draw[dashed] (0,0)--(1,1);
%\draw[dashed] (1,0)--(1,1);
 \draw[scale=1,domain=0:1,smooth,variable=\x,blue] plot ({\x},{-0.75*\x+1});
  \draw[scale=1,domain=0:1,smooth,variable=\x,red] plot ({\x},{0.33*\x+0.266});
    \draw[magenta, scale=0.5] (0,0) sin (0.5,2); 
     \draw[magenta, scale=0.5] (0.5,2) cos (2,0);
    
   % \draw(0,0) cos (1,2) cos (3,5) cos (5,0);
\end{tikzpicture}  
  \caption{Example~\ref{e.AFinterior}, one dimensional illustration (where $k=2$ and $\ell=1$).}
   \label{f.eAFinterior}
\end{figure}
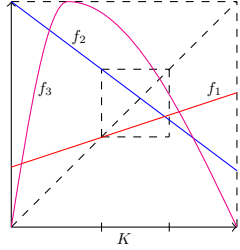

 To see that $K\subset \overline{A_F}$ consider the skew product  $G\colon \Sigma_{k} \times M\to \Sigma_{k} \times M$  with fiber maps $f_1,\dots, f_{k}$. By the Hutchinson theory \cite{Hu}, it follows that  $A_G=K$. 
 Since $A_G\subset A_F$ it follows $K\subset \overline{A_F}$.
 
 Figure~\ref{f.eAFinterior} illustrates the example in the one dimensional case.

 }}
\end{example}

%
% In these examples the closure of the target set, $\overline{A_{F}}$, is a union of disjoint intervals.
%Since attractors of hyperbolic systems are totally disconnected or connected sets, see \cite{??}, then 
%it would not be possible to obtain the same example by means of the 
%\textcolor{red}{Hutchinson theory,} see 
%Remark ?? for a detailed discussion.

\subsection{Disconnected spines}
Let $\mathcal {S}$ denote the set of  step skew products $F \colon \Sigma_{k}\times [0,1]
\to  \Sigma_{k}\times [0,1]$ over the shift map with  $C^1$ fiber maps. In $\mathcal S$ 
 we consider the following metric: given $F$ and $G$ in $\mathcal{S}$ with fiber maps $f_{1},\dots, f_{k}$ and 
 $g_{1},\dots, g_{k}$, respectively, let 
$$
d(F, G)\eqdef \max_i d_{C^1} (f_i, g_i), 
$$
where $d_{C^1}$ denotes the $C^1$ uniform distance.

The next result shows that  the fiber structure of attracting graphs  
 can be
  much more complicated than the fiber structure 
  of the maximal attractors. 

\begin{mteo}
[Disconnected spines]
\label{mt.espinhodesconexo}
For every $m \geq 1$ 
there is an open set $\mathcal{S}_m\subset\mathcal{S}$ of step skew products 
$
F\colon \Sigma_{4}\times [0,1]
\to  \Sigma_{4}\times [0,1]
$
with injective fiber maps
such that 
every $F\in \mathcal{S}_m$ has the following properties: 
%let $\mathfrak{b}$ be a Bernoulli probability with no trivial weights in $\Sigma_4$,
% let $\mathfrak{b}$ be a Bernoulli probability with no trivial weights in $\Sigma_4$,

 \begin{itemize}
 \item[(1)] For every disjunctive point $z$  it holds 
$\overline{\mathrm{graph}\,\rho}=\omega_F(z)$.
 \item[(2)]
 There is a dense subset 
 $\Gamma$
 of $\Sigma_{4}$  such that
 for every sequence $\vartheta\in \Gamma$ the intersection of the fibre $\{\vartheta\}\times [0,1]$ with 
 $\overline{\mathrm{graph}\,\rho}$ is the
 union of $m$ disjoint nontrivial intervals. 
  \end{itemize}
\end{mteo}

\begin{remark}\emph{
For step skew products $F$ as in Theorem \ref{mt.espinhodesconexo} the target set $\overline{A_{F}}$ is
 a finite union of disjoint intervals, see the proof of Theorem~\ref{mt.espinhodesconexo}. 
Recall that, in the case of uniformly contracting invertible maps, compact sets 
 that are Hutchinson attractors  are either a connected set or a totally disconnected set.
Hence $\overline{A_{F}}$ is not the Hutchinson attractor of any hyperbolic iterated function systems.
This shows that to understand the structure of graphs we need to
 go beyond the Hutchinson theory.}
 \end{remark}
%
%In the case the fiber maps are contraction we have that the set $\overline{A_{F}}$ coincide 
%with the Hutchinson attractor of some IFS. We recall that any Hutchison attractor is either 
%connected or totally disconnected. In this case, for every $\vartheta$ we have that 
%${\vartheta}\times $ is either a connected set or a subset of a totally disconnected set.
%This fact motivated us to define the target set in order to understand the fiber 
%structures of invariant graphs because  the topological structure of $\overline{A_{F}}$
%is not restricted to the dichotomy above. Hence this allows us to construct new examples 
%of invariant graphs with a complicated fiber structure, which can not be constructed 
%just applying the Hutchison theory. 
%To understand the structure of graphs we need to go beyond Hutchinson theory.}

Let us observe that Theorem~\ref{mt.espinhodesconexo} was partially motivated
by the questions posed in \cite{Yuri} as well as his constructions of $\mathcal{K}$-pairs, see 
Section~\ref{ss.kpairs}.

\begin{remark}
{\em{Kudryashov  poses  the following \cite[Question 2.9.3]{Ku2} about the fiber
structure of attractors: ``Does the likely limit set of  typical step skew products over a Markov shift intersect each fiber on a finite union of intervals and points?''. Theorem~\ref{mt.mixing} allows us to address this question in a perhaps more transparent way. In view of item (2) and the characterization of
the maximal attractor in Proposition~\ref{p.maximalattractor} we have that
$$
\overline{\mathrm{graph}\,\rho}= \bigcup_\vartheta \{\vartheta\} \times I_\vartheta^A,
\quad \mbox{where} \quad I_\vartheta^A\eqdef \bigcap_{n\ge 1} f_{\vartheta_{-1}} \circ  \cdots \circ f_{\vartheta_{-n}}(\overline{A_{F}}).
$$
Thus, to understand the fiber structure of $\overline{\mathrm{graph}\,\rho}$ it is enough to study the sets
$ I_\vartheta^A$.
} }
\end{remark}

\subsection{Attracting invariant  measures}
We  call $F$ a
step skew product over the measure-preserving shift map $(\Sigma_{k},\mathscr{F},\mathbb{P},\sigma)$ 
to emphasize that the base map $\sigma$ is a measure-preserving transformation of some probability 
space $(\Sigma_{k},\mathscr{F},\mathbb{P})$. The sigma-algebra $\mathscr{F}$  is always assumed to be
the product sigma-algebra $\mathscr{E}^{\mathbb{Z}}$, where $\mathscr{E}$ is the discrete sigma-algebra of  $\{1,\dots, k\}$. If $\mathbb{P}$ is a Markov measure (see Definition below) we say that $F$ 
is a step skew product over a Markov shift.

Associated to the coding map $\rho$
defined in  \eqref{e.varrhoproj} we consider the map 
\begin{equation}\label{ultst}
\hat\rho \colon S_{F}\to S_{F}\times M, \quad
\hat  \rho(\vartheta)\eqdef (\vartheta,\rho(\vartheta)),
\end{equation}
 whose image is the graph of the map $\rho$ (denoted by $\mathrm{graph\, \rho}$).

For the precise definitions of isomorphic  measure-preserving transformations
 see Definition~\ref{d.isomorphic}. In what follows, given a measure $\mu$ we denote by
  $\mathrm{supp}\, \mu$ its support.

\begin{mteo} [Global attractor of measures]
\label{mt.skewproduct}
Let $M$ be a compact metric space and let $F\colon \Sigma_{k}\times M\to  \Sigma_{k}\times M$ be a step skew product over a measure-preserving shift map  $(\Sigma_{k},\mathscr{F},\mathbb{P},\sigma)$.
Assume that $\mathbb{P}(S_{F})=1$. 
Then:
\begin{itemize}
\item[(1)] 
The measure $v_{\mathbb{P}}=\hat \rho_\ast \mathbb{P}$ is the 
  global attractor of 
  $F_{*|\mathcal{M}_{\mathbb{P}}}$. In particular, $F$
  has a unique invariant measure with marginal $\mathbb{P}$
  whose 
 disintegration 
with respect to $\mathbb{P}$ is the Dirac delta measure $\delta_{\rho(\cdot)}$.
\item[(2)]
$(F,v_{\mathbb{P}})$ is isomorphic to 
  $(\sigma,\mathbb{P})$. 
  \item[(3)]
  $\mathrm{supp}\, v_{\mathbb{P}}=\overline{\mathrm{graph} \,\,(\rho_{|\mathrm{supp}\,\mathbb{P}}})$.
%  \item 
%   For $\mathbb{P}$-almost every sequence $\xi$ 
% and every point $x\in X$ it holds 
% $$
% \frac{1}{n}\sum_{i=0}^{n-1}\delta_{F^{i}(\xi,x)}  {\to} v_{\mathbb{P}},
% $$
% where the convergence is considered 
% in the weak$\ast$ topology.
\end{itemize}
\end{mteo}
 
 \begin{remark}\label{r.stark}
 {\em{
 Note that in setting considered by Stark in \cite{Stark} with negative maximal Lyapunov exponent\footnote{The {\emph{maximal Lyapunov exponent}} is the rate of growth of the Lipschitz constant of the compositions $f_\vartheta^n$.} one has $\mathbb{P}(S_F)=1$.
 Thus Theorem~\ref{mt.skewproduct} holds in this setting. As far as we know, the study of ergodic properties for these skew products is still quite incipient and there are very few available results, see \cite{SS} and also \cite{FQ} for some specific cases. We refer to \cite{Keller} for a bifurcation setting.}}
 \end{remark}

 For step skew products whose fiber maps are defined on a subset $M$ of $\mathbb{R}^{m}$  
  Theorem \ref{mt.skewproduct} can be improved for a certain class of step skew products over a Markov shift having a splitting property. 
As stated below, under these conditions, the convergence to the attracting measure  in Theorem~\ref{mt.skewproduct} is exponentially fast.

  More precisely, consider the {\emph{Wasserstein metric}} in $\mathcal{M}_{1}(\Sigma_{k}\times M)$ defined by
 \begin{equation}\label{Wass}
  d_W (\mu, \nu) \eqdef\sup \left\{ \left|\int \phi\, d\mu - \int \phi\, d\nu \right|,\, \, \phi \in \mathrm{Lip}_1 (  \Sigma_k\times M) \right\},
  \end{equation}
 where $\mathrm{Lip}_1 (\Sigma_k\times M)$ denotes the space of Lipschitz maps with Lipschitz constant one 
 ($\Sigma_k\times M$ is endowed with the canonical metric, see Section \ref{ss.preliminaries}).
Observe that  this  metric generates the weak-$*$ topology in 
$\mathcal{M}_{1}(\Sigma_{k}\times M)$
see \cite{dW}.
% We consider the 
 %metric $d_{1}(x,y)=\sum |x_{i}-y_{i}|$.  
 
Further, recall that
a $k\times k$ matrix  $P=(p_{ij})$ is a 
\emph{transition matrix} if 
$p_{ij}\ge 0$ for all $i,j$ and 
for every $i$ it holds
$\sum_{j=1}^{k}p_{ij}=1$.
A \emph{stationary probability vector} associated to $P$  is a 
  vector $\bar p=(p_{1},\ldots,p_{k})$ whose elements are 
  non-negative real numbers, sum up to $1$, and satisfies $\bar p \, P= \bar p$. 
  Given a transition matrix $P$  and a stationary probability vector $\bar p$, there is a unique probability measure $\mathbb{P}$ 
on $\Sigma_{k}$ such that the sequence of coordinate mappings on $\Sigma_{k}$
is a homogeneous Markov chain with probability transition 
$P$ and starting probability vector $\bar p$.  For details see, e.g.,
\cite[Chapter 1]{Revuz}.
 The measure $\mathbb{P}$ is the \emph{Markov measure} associated to the pair $(P,\bar{p})$.
  The measure $\mathbb{P}$ is $\sigma$-invariant and the measure-preserving dynamical systems $(\Sigma_{k},\mathscr{F},\mathbb{P},\sigma)$ is a \emph{Markov shift}.
  A Markov shift $(\Sigma_{k},\mathscr{F},\mathbb{P},\sigma)$ 
with transition matrix $P=(p_{ij})$
is called \emph{irreducible} if  
for every $\ell,r\in \{1,\dots,k\}$ there is $n=n(\ell,r)$ such that
$P^{n}=(p^n_{ij})$ satisfies $p^n_{\ell,r}>0$.
We observe that an irreducible transition matrix has a unique positive stationary
probability vector $\bar p=(p_{i})$, see 
\cite[page 100]{Kemeny}.
 Finally, recall that a sequence $(a_{1}\dots a_{\ell})\in \{1,\dots,k\}^{\ell}$ is \emph{admissible} with respect to $P=(p_{ij})$ if $p_{a_{i}a_{i+1}}>0$ for every $i=1,\dots, \ell-1$. 
 
\begin{defi}[Splitting condition]
\label{splitgene2}\emph{
Let $F$ be a step skew product over 
 a Markov shift $(\Sigma_k, \mathscr{F}, \mathbb{P}, \sigma)$ with fiber maps $f_{1},\dots ,f_{k}\colon M\to M$.
 Consider
the projections $\pi_s \colon M \to \mathbb{R}$ given by
$\pi_s(x_1,\dots, x_m)\eqdef x_s$, where $s\in \{1,\dots, m\}$. We say that $F$
 \emph{splits} if the are $\mathbb{P}$-admissible sequences
$(a_{1}\dots a_{\ell})$ and $(b_{1}\dots b_{r})$ with $a_\ell=b_r$
such that the sets $M_{1}\eqdef
 f_{a_{\ell}}\circ \dots \circ f_{a_{1}}(M)$ and $M_{2}\eqdef f_{b_{r}}\circ \dots \circ f_{b_{1}}(M)$ satisfy
$$
\pi_{s}(f_{\vartheta}^{n}(M_{1}))\cap \pi_{s}(f_{\vartheta}^{n}(M_{1}))=\emptyset,
$$
for every $\vartheta$, every $n$, and every $s$.}
\end{defi}

\begin{mteo}[Exponential convergence to the attracting measure]\label{exponentialdw}
Let $M\subset \mathbb{R}^{m}$ be a compact subset and
 $F\colon \Sigma_{k}\times M\to \Sigma_{k}\times M$ be a step skew product over 
 an irreducible Markov shift $(\Sigma_k, \mathscr{F}, \mathbb{P}, \sigma)$. Suppose that $F$ splits. Let $P=(p_{ij})$ be the corresponding transition matrix 
and suppose that there is $u$ such that $p_{uj}>0$ for every $j$.
Then  
  \begin{itemize}
  \item[(1)]
  The coding map $\rho$ is defined $\mathbb{P}$-almost everywhere, i.e., $\mathbb{P}(S_{F})=1$. 
  \item[(2)]
  There are $\lambda\in (0,1)$ and $N\ge 0$ such that
    for every probability measure $\mu \in  \mathcal{M}_\mathbb{P}  (\Sigma_{k}\times M) $
   we have
  $$
 d_{W}(F_{*}^{n}\mu, 
 \hat \rho_{*}\mathbb{P}
%\mu_{\mathbb{P}}
 )\le m\, \lambda^n, \quad \mbox{for every $n\ge N$}.
 $$
 \item [(3)]
 There are $q\in (0,1)$ and $N\ge 0$ such that
for $\mathbb{P}$-almost every $\vartheta\in \Sigma_k$ 
there is $C(\vartheta)>0$ such that
$$
d(F^{n}(\vartheta,x),\hat\rho(\sigma^{n}(\vartheta)))\le C(\vartheta) q^n,
\quad \mbox{for every $n\ge N$ and  every $x\in M$}.
$$
\end{itemize}  
\end{mteo}

Let us observe that the second item of Theorem~\ref{exponentialdw} is a result in the same spirit as the one in 
\cite{BhLe}
in  the context of
i.i.d. random products. In \cite{BhLe},
assuming also a splitting-like condition, it is proved  that the Markov operator has  a 
unique stationary measure that attracts 
 probability measures
exponentially fast.

\begin{example}[$m$ dimensional examples]
\label{e.msplits}
{\em{
Consider a pair of  increasing\footnote{This means that for every $(x_1,\dots,x_m)\in [0,1]^m$, every 
$s,i\in \{1, \dots, m\}$, and every $j\in \{1,2\}$, the map $x\mapsto \pi^s (f_j(x_1,\dots, x_{i-1},x,x_{i+1},\dots,x_m))$ is increasing.}
 continuous maps $f_1,f_2\colon [0,1]^m \to [0,1]^m$ such that
\begin{itemize}
\item
$(0^m)$ is a fixed point of $f_1$ which is a global attractor of $f_1$ in $[0,1]^m\setminus \{(1^m)\}$,
%\item
%$(1^m)$ is a fixed point of $f_1$,
\item
$f_2$ is an uniform contraction whose fixed point $x_1$ is in $(0,1)^m$, and
\item
$\pi_s(x_1) \ne \pi_s( f_1(x_1))$ for every $s=1,\dots,m$.
\end{itemize}
Note that  the fiber maps have a  ``mixed'' behavior  with intermingled regions of
 expansion and contraction.

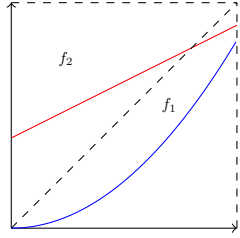
\begin{figure}[h!]
\centering
\begin{tikzpicture}[xscale=3,yscale=3]
 \node[scale=0.6,left] at (0.3,0.75) {$f_{2}$};
  \node[scale=0.6,below] at (0.7,0.6) {$f_{1}$};
\draw[->] (0,0)-- (1,0);
\draw[->] (0,0)--(0,1);
%\draw[red,-]  (0,0)--(0.6,0.2);
%\draw[red,-]  (0.6,0.2)--(1,0.75);
%\draw[blue,-]  (0,0.15)--(0.4,0.8);
%\draw[blue,-]  (0.4,0.8)--(1,1);
\draw[red,-] (0,0.4)--(1,0.9);
\draw[dashed] (0,1)--(1,1);
\draw[dashed] (0,0)--(1,1);
\draw[dashed] (1,0)--(1,1);
 \draw[scale=1.2,domain=0:0.83,smooth,variable=\x,blue] plot ({\x},{\x*\x});
\end{tikzpicture}
 \caption{The fiber maps $f_1$ and $f_2$ of Example~\ref{e.msplits} for $m=1$.}
 \label{f.splits}
\end{figure}

 Consider a Markov measure $\mathbb{P}$ whose  transition matrix $P$ is of the form
\begin{equation}
\label{e.matriztran}
 \left( \begin{matrix} p_{11} & 1-p_{11}\\p_{21} & 1-p_{21} 
\end{matrix}
\right),
\quad
p_{11} \in [0,1], \, p_{21}\in(0,1).
\end{equation}
Then the skew product $F\colon \Sigma_2 \times [0,1]^m \to \Sigma_2 \times [0,1]^m$ splits (with respect $\mathbb{P}$), hence the conclusions of Theorem~\ref{exponentialdw} hold.

To see that $F$ splits
note that $f_2^k ([0,1]^m)\to \{x_1\}$. Hence, for sufficiently large $k$, the sets
$f_1\circ f_2^k([0,1]^m)$ and $f_2\circ f_2^k([0,1]^m)$ are disjoint. 
Thus $f_2\circ f_1\circ f_2^k([0,1]^m)$ and $f_2^{k+2}([0,1]^m)$ are also disjoint. 
To conclude observe that the finite 
sequences  $(212\dots 2),(2\dots 2)\in\{1,2\}^{k+2}$ are both admissible and that the fiber maps $f_1,f_2$ are increasing 
monotones.

If the transition matrix is of the form
$$
 \left( \begin{matrix} p_{11} & 1-p_{11}\\0 & 1 
\end{matrix}
\right)
$$
the previous construction holds if, for instance, $f_1\circ f_2$ is a uniform contraction whose  fixed point is in
$(0,1)^m$. More complicated cases can be considered by using the arguments in \cite[Section 2.4.2]{DM} dealing with ``monotone maps'' in $\mathbb{R}^m$.

We observe that in the one dimensional case  different ``geometrical' configurations'' 
for the fiber maps can be considered, see
the discussion in Section~\ref{ss.examples}.
}}
\end{example}

 \subsection{Milnor attractors}
 \label{ss.milnor}

We now see some applications of our results to the study of Milnor attractors.  For that we need some definitions.

Consider a metric space $M$, a probability measure $ \mathfrak{s}$ on $M$,  
and a continuous map $f\colon M\to M$.
The
   \emph{realm of attraction}
of a set $B \subset M$, denoted by $\delta(B)$, is the set  of points 
 $z$ such that $\omega_{f}(z)\subset B$. 
 A closed set $X$ is a \emph{Milnor attractor} (with respect to $f$ and $\mathfrak{s}$)
 if it satisfies 
$ \mathfrak{s}(\delta(X))>0$
 and
 $\mathfrak{s}(\delta(Y))<\mathfrak{s}(\delta(X))$
 for every closed set $Y$ with $Y\subsetneq X$. 
% Finally, 
%the \emph{likely limit set} $A_{M}$  (with respect to $f$ and $\mathfrak{s}$) 
%is the smallest compact subset of 
%$X$
%such that $\omega_f (z) \subset A_M$ for 
%$\mathfrak{s}$-almost every point
% $z\in X$.

 \begin{mcoro}\label{c.milnor1}
  Let $F$ be as in Theorem~\ref{mt.espinhodesconexo}.  
 Consider an ergodic measure $\mathbb{P}$ fully supported on $\Sigma_k$
 and a measure $\mu$ with marginal   $\mathbb{P}$.
 Then 
 $\overline{\mathrm{graph}\,\rho}$ is the unique Milnor attractor of $F$ with respect to $\mu$.
 \end{mcoro}

 To see why the corollary  is so let
$D$ be the space of disjunctive 
 sequences of $\Sigma_k$ and $S\eqdef D\times M$. 
 Consider the disintegration  $(\mu_\vartheta)_{\vartheta \in \Sigma_k}$ of  $\mu$ and 
 observe that
 $$
 \mu (S)= \int_{\Sigma_k}   \mu_\vartheta (S^\vartheta)\,  d \mathbb{P} (\vartheta),
\quad
\mbox{where $S^\vartheta=\{ x\in M \colon (\vartheta,x)\in S\}$}.
$$
 Note that $S^\vartheta=M$ if $\vartheta$ is disjunctive and that  $\mathbb{P}(D)=1$ (by ergodicity of
 $\mathbb{P}$).   Thus $ \mu (S)=\int_{D}   d \mathbb{P} (\vartheta)= \mathbb{P}(D)=1$.
Item (1) of Theorem~\ref{mt.espinhodesconexo} implies that $\overline{\mathrm{graph}\,\rho}=
\omega_F(z)$ for every $z\in S$.
 These two claims imply that
  $\overline{\mathrm{graph}\,\rho}$ is the unique Milnor attractor (with respect to $F$ and $\mu$).

 \begin{mcoro}\label{c.milnor2}
 Consider $F$ and $\mathbb{P}$ 
 as in  Theorem~\ref{exponentialdw} and 
  a measure $\mu$ with marginal   $\mathbb{P}$.
%
% the product measure  $\mathfrak{s}=\mathbb{P}\times \mathfrak{m}$,
%where  $\mathfrak{m}$ is  any probability measure on $M$. 
 Then $\overline{\mathrm{graph}\,\rho}$ is the unique Milnor attractor (with respect $F$ and $\mu$).
 \end{mcoro}

  By the third item
 in Theorem~\ref{exponentialdw},
  there is
  $\Omega\subset \Sigma_k$ with $\mathbb{P}(\Omega )=1$ such that
  $\omega_F(z)= \overline{\mathrm{graph}\,\rho}$ for all $z\in \Omega \times M$.
  Arguing as in the proof of Corollary~\ref{c.milnor1},  we get $\mu (\Omega \times M)=1$ and thus 
   $\overline{\mathrm{graph}\,\rho}$ is the unique Milnor attractor of $F$ with respect $\mu$.

\medskip

 Special cases of the corollaries above are the product measures $\mu=\mathbb{P}\times \mathfrak{m}$,
 where $\mathfrak{m}$ is  any probability measure on $M$.

\subsection{Prototypical examples}
\label{ss.examples}
We now explain a class of examples that motivate the ones in this paper and also provide families of 
``non-contracting'' examples where our results can be applied. In some sense, this section is a continuation of Example~\ref{e.msplits}.

In \cite{DiGe} are introduced the so-called ``porcupine-like horseshoes'' (these horseshoes were also motivated by the so-called bony attractors, see \cite{Yuri}). The simplest model 
of such horseshoes is obtained considering a step skew product map $F\colon \Sigma_2\times [0,1]
\to \Sigma_2 \times [0,1]$ whose fiber maps are $C^1$ and injective and are as depicted in Figure~\ref{f.fiberporc}. It turns out that,
for appropriate choices of  the fiber maps $f_1$ and $f_2$, the maximal attractor $\Lambda_F$ of $F$ is a ``non-hyperbolic'' transitive set\footnote{Obviously, here the hyperbolicity concerns only the one dimensional fiber direction.},  whose subsets of fiber contracting and fiber expanding periodic points are both dense
(indeed, the maximal attractor is an important type of transitive set called a ``homoclinic class'').
As in Example~\ref{e.msplits}, a  key feature of these systems is their ``mixed'' behavior in the fiber direction.

 \begin{figure}[h!]
 \centering
 \begin{tikzpicture}[xscale=3,yscale=3]
 \draw[->] (0,0)-- (1,0);
 \draw[->] (0,0)--(0,1);
 \draw[red,-] (0,0.9)--(1,0);
 \draw[dashed] (0,0)--(1,1);
   \draw[green,smooth,samples=100,domain=0.0:1] plot(\x,{-\x*\x+2*\x});
  \node[scale=0.6,left] at (0.6,0.9) {$T_{2}$};
   \node[scale=0.6,below] at (0.8,0.4) {$T_{1}$};
 %\draw[green] (0,0) arc (-3:-60:-1);
 \draw[dashed] (0,1)--(1,1);
 \draw[dashed] (1,0)--(1,1);
 \end{tikzpicture}
 \caption{ The fiber maps of a porcupine-like horseshoe.}
 \label{f.fiberporc}
 \end{figure}
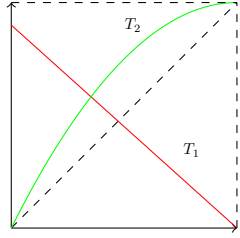
 
%
%\begin{figure}[!h]
%   \caption{The fiber maps $f_0$ and $f_1$}\label{f.fiberporc}
%     \end{figure}

In the setting of porcupine-like horseshoes  above, the set 
 $S_F$ is a residual subset of $\Sigma_2$ whose complement $S_F^c$ is dense, see \cite{DiGe}\footnote{In \cite{DiGe} the set $S_F$ corresponds to the ``subset of sequences in $\Sigma_2$ with trivial spines''. Indeed, this is a simple remark that can be obtained in more general settings: it is not difficult to see that if  $S_F\ne \emptyset$ then $S_F$ is a residual subset of $\Sigma_k$.}.
 In \cite{Yuri} the structure of the sets $S_F$ and $S_F^c$ is analyzed more closely,  
 proving that 
 $\rm{HD} (S_F^c)<2$  (here $\rm{HD}$ stands for the Hausdorff dimension), see \cite[Lemma 4]{Yuri},
 and that 
 $\mathfrak{b}_{p} (S_F)=1$ for every $p\in (0,1)$, here $\mathfrak{b}_{p}$ is the Bernoulli measure with weights
 $(p,1-p)$, $p\in (0,1)$, see \cite[Theorem 1]{Yuri}. Arguing exactly as in Example~\ref{e.msplits}, one can prove that $F$ satisfies the splitting property for any Markov measure whose transition matrix is as in \eqref{e.matriztran}.
  Therefore,
 we get   a class of ``non-contracting''   examples satisfying the hypotheses in Theorems~\ref{mt.mixing} and \ref{mt.skewproduct}.
 
 Finally, let us observe that \cite{Yuri} (mainly) deals with systems whose fiber maps preserve the orientation. In his context it holds $A_F=[0,1]=\overline{A_F}$. It is interesting to compare with the case of porcupines, where $A_F 
 \subsetneq [0,1]$ and 
 $\overline{A_F}=[0,1]$.

\subsection{Notation}
In the space $\Sigma_k$ of two-sided infinite sequences with $k$ symbols we
consider
the canonical metric, which generates the 
product topology, defined by 
\begin{equation}
\label{e.do}
 d_{0}(\vartheta,\xi)\eqdef \frac{1}{2^{n}}, \quad \mbox{where}\quad n=\min\{|i|\colon \vartheta_{i}\neq \xi_{i}\}.
\end{equation}

Among all open sets, there is a special family called \emph{cylinders} which is a base for the 
product topology defined as follows.  Given a finite sequence $(a_{1}, \dots, a_{\ell})\in \{1,\dots, k\}^\ell$ and
$m\in \mathbb{Z}$ we let
\begin{equation}
 \label{e.generalcyl}
 [m;a_{1}\dots a_{\ell}]\eqdef \{\vartheta\in \Sigma_k\colon \vartheta_{m}=a_{1},\dots,\vartheta_{m+\ell-1}=a_{\ell}\}.
 \end{equation}
 Note that what determines in which position goes the sequence $a_{1}\dots a_{\ell}$ is 
only the index $m$. 

We also consider the space of unilateral sequences 
$\Sigma_{k}^{+}=\{1,\dots,k\}^{\mathbb{N}}$ endowed with the product topology and define cylinders in the obvious way.

Finally, given $\vartheta \in \Sigma_k$ we write $\vartheta=\vartheta^-.\vartheta^+$, where $\vartheta^+\in \Sigma_k^+$ is such that $\vartheta^+_i=\vartheta_i$, $i\ge 0$.

\section{Transitive invariant sets. Proof of Theorem~\ref{mt.mixing}}
\label{s.Milnor}

\subsection{Preliminary topological property} 
\label{s.maximalatractor}
%and Corollary \ref{mc.maximalcorollary}.
Let $F\colon \Sigma_{k}\times M\to \Sigma_{k}\times M$ be a step skew product over the shift map
as in \eqref{defiskew2}. 
The  {\emph{spine}}  $I_\vartheta$ of a sequence $\vartheta\in\Sigma_{k}$ is defined 
 as follows,
 \begin{equation}
 \label{e.defspine}
  I_\vartheta \eqdef \bigcap_{n\ge 1} f_{\vartheta_{-1}} \circ  \cdots \circ f_{\vartheta_{-n}}(M).
 \end{equation}
  The next simple proposition characterizes the maximal attractor  
  $\Lambda_F$  of $F$
  in term of 
 the set of spines.

 %Recall the definition of $X^*$ in \marginpar{completar}
 
 \begin{prop}
  \label{p.maximalattractor}
$
\Lambda_F
\eqdef 
\displaystyle{
\bigcap_{n\geq 0}
F^{n}(\Sigma_k \times M)}=
{\displaystyle{\bigcup_{\vartheta\in \Sigma_{k}}\{\vartheta\}\times I_{\vartheta}.}}
%\subset \Sigma_{k}\times X^{*}.
$
\end{prop}
\begin{proof}
Consider a point $(\vartheta,x)$ such that $x\in I_{\vartheta}$. Then  
for every $n\geq 1$ we have that  the set 
$f_{\vartheta_{-n}}^{-1}\circ \cdots \circ f_{\vartheta_{-1}}^{-1}(x)\subset M$
is well defined and nonempty, hence
$$
(\vartheta,x)=F^{n}\big( \sigma^{-n}(\vartheta),  
f_{\vartheta_{-n}}^{-1}\circ \cdots \circ f_{\vartheta_{-1}}^{-1}(x)\big).
$$
This implies that $(\vartheta,x)\in \Lambda_F$, proving the inclusion  ``$\supset$''.
For the inclusion  ``$\subset$'' take any $(\vartheta,x)\in \Lambda_F$. By definition of $\Lambda_F$, 
for every $n\geq 1$ there is $({\vartheta}(n) ,{x}(n))\in \Sigma_{k}\times M$
 such that $F^{n}({\vartheta}(n), {x}(n) )=(\vartheta,x)$. Noting that ${\vartheta}(n)=\sigma^{-n}(\vartheta)$ we conclude that 
$$
x\in f_{\vartheta_{-1}}\circ  \cdots \circ f_{\vartheta_{-n}}(M).
$$
Since this holds for every $n\geq 1$ it follows that $x \in I_{\vartheta}$, proving
the proposition.
\end{proof}

\subsection{Proof of Theorem~\ref{mt.mixing}}
\label{sss.m.tmixing}
Recall the definitions of 
 the coding map 
 $\rho\colon S_{F}\rightarrow X$ 
 in  \eqref{e.varrhoproj} and of the target set $A_{F}=\rho (S_{F})$.
By definition,  
$F(\overline{\mathrm{graph}\, \rho})\subset \overline{\mathrm{graph}\, \rho}$ and thus 
$\overline{\mathrm{graph}\, \rho}\subset \Lambda_F$.
Note also that $\mathrm{graph}\, \rho\subset S_{F}\times A_{F}.$

Denote by $\mathcal{D}$ the set of disjunctive sequences of $\Sigma_k$.
The first part of item (1) of the theorem follows immediately from the next lemma. 

\begin{lema}\label{l.densedisjunctive}
Let $z=(\vartheta,x)\in \mathcal{D}\times M$. Then
$\overline{\mathrm{graph}\,\rho}\subset \omega_F(z)$.
\end{lema}

\begin{proof}
Since $\omega_F(z)$ is closed it is enough to see that $\mathrm{graph}\,\rho\subset \omega_F(z)$.
Take any point $(\beta,\rho(\beta))\in \mathrm{graph}\,\rho$ and consider a sequence $n_{\ell}\to \infty$
(with $n_\ell > \ell$)
such that 
$\sigma^{n_{\ell}}(\vartheta)\in [-\ell;\beta_{-\ell}\dots\beta_{\ell}]$
 for every $\ell\geq 0$ (here we use that $\vartheta\in \mathcal{D}$).  
 Note that $\sigma^{n_{\ell}-\ell}(\vartheta)\in [0;\beta_{-\ell}\dots\beta_{\ell}]$.
 Let
 $$
 x_\ell \eqdef f_{\vartheta_{n_\ell-\ell}} \circ \cdots \circ f_{\vartheta_0}(x).
 $$
 By definition of $S_{F}$ and since
 $\beta \in S_{F}$ it follows that
  $$
  \lim_ {\ell \to 
 \infty} f_{\beta_{-1}}\circ\dots \circ f_{\beta_{-\ell}}(x_\ell)=\rho(\beta).
$$ 
 Hence
 $\lim_{\ell \to \infty} F^{n_{\ell}}(\vartheta,x)=(\beta,\rho(\beta))$, proving that
  $\mathrm{graph}\,\rho\subset \omega_{F}(z)$.
  \end{proof}
  
We postpone the proof of the second part of the first item of Theorem~\ref{mt.mixing} to the end of this section.

 To prove the remainder items in the theorem
we need 
two preparatory lemmas.
% In what follows, we consider the space of unilateral sequences 
%$\Sigma_{k}^{+}=\{1,\dots,k\}^{\mathbb{N}}$ endowed with the product topology. 
Given
 $\xi^+=(\xi_0^+\xi_1^+\dots)\in \Sigma_{k}^{+}$ we consider the subset  of $\Sigma_{k}$ defined by
$$
[\xi^{+}] \eqdef
\{\vartheta \in \Si_k \colon \vartheta_i =\xi^+_i \quad  \mbox{for every $i\ge 0$}\}.
$$

\begin{lema}
\label{l.interesecao}
Consider any open set $[-\ell;b_{-\ell}\dots b_{\ell}]\times J\subset \Sigma_{k}\times M$ such that 
$$
A_{F}\cap  \big( 
(f_{b_{-1}}\circ\cdots \circ f_{b_{-\ell}})^{-1}( J) \big)\neq \emptyset.
$$
Then for every open set $Q\subset \Sigma^{+}_{k}$
there is $n_{0}=n_0(Q)\ge 0$ such that for every $n \geq n_{0}$ 
there is a sequence $\xi^{+,n}\in Q$ such  that
$$
F^{n}\big( [\xi^{+,n}]\times M \big)
\subset  [-\ell;b_{-\ell}\dots b_{\ell}]\times J.
$$ 
\end{lema}

\begin{proof}
  By hypothesis there is $q\in (f_{b_{-1}}\circ\cdots \circ f_{b_{-\ell}})^{-1}( J)
  \cap A_{F}$.   Thus, by definition of $A_{F}$,
  there is a sequence $c_{1},\dots c_{p}$ such that 
   $$
   f_{c_{1}}\circ \cdots \circ f_{c_{p}}(M)\subset (f_{b_{-1}}\circ\cdots \circ f_{b_{-\ell}})^{-1}( J).
 $$ 
 Hence
\begin{equation}\label{e.1}
  f_{b_{-1}}\circ\cdots \circ f_{b_{-\ell}}\circ f_{c_{1}}\circ\cdots\circ f_{c_{p}}(M)\subset J.
\end{equation}

Let $\sigma^{+}$ be the shift map on $\Sigma_{k}^{+}$ and recall that
it is topologically mixing. This implies that there is $m_{0}>0$ 
such that
 for each $n\geq m_{0}$  there is a sequence $\vartheta^{+,n}\in Q$ 
 such that 
 \begin{equation}\label{e.2}
 (\sigma^{+})^{n}(\vartheta^{+,n})\in [0; c_{p}\dots c_{1}b_{-\ell}\dots
 b_0 \dots  b_{\ell}] \subset \Sigma_k^+.
 \end{equation}
 For  each $n$ with $n-\ell-p \geq m_{0}$ we define the  sequence
 $$
 \xi^{+,n}\eqdef \vartheta^{+,n-(\ell+p)}\in \Sigma_{k}^{+}.
 $$
%Since $n-(\ell+p)\ge m_0$ this sequence is well defined.
Take 
 any $\zeta \in [\xi^{+,n}]$ and note that, by construction, for every  $x\in M$ we get
$$
 F^{n-\ell-p}(\zeta,x)\in [0;c_{p}\dots c_{1}b_{-\ell}\dots b_0 \dots b_{\ell}]\times M.
 $$
 Therefore from equations \eqref{e.1} 
 and \eqref{e.2} it follows that
  $$
 F^{n}(\zeta,x)=F^{\ell+p}( F^{n-\ell-p}(\zeta,x))
 \in [-\ell;b_{-\ell}\dots b_{\ell}]\times J,
 $$
proving the lemma.
\end{proof}

%  We now prove the second assertion of the theorem. Suppose that the set $\overline{A_{\mathrm{t}}$
%  is a Conley attractor.  Thus there is a open neighbourhood $U$ such that 
%   
% $\mathcal{B}$ has a unique fixed point. We need to see that $
%\overline{\mathrm{graph}\, \varrho}= \Lambda_F.$ Note that the inclusion 
%``$\subset$'' follows from the invariance of the graph. 
%For the inclusion ``$\supset$'' we need to see that $\mathrm{graph}\,\varrho$ is dense in $\Lambda_F$. 
%For that we see that for a given point $(\xi,s)\in \Lambda_F$ any 
%open neighbourhood $V=[-\ell;\xi_{-\ell}\dots \xi_{\ell}]\times J$ of it intersects 
% $\mathrm{graph}\,\varrho$. 

\begin{lema}\label{l.invariancep}
Let $\Omega$ be a compact subset of $\Sigma_k\times \overline{A_{F}}$
with $F(\Omega)=\Omega$. Then $\Omega\subset\overline{\mathrm{graph}\,\rho}$.
\end{lema} 

\begin{proof}
Take any point $(\vartheta,y)\in  \Omega$ and
 $V=[-\ell;\vartheta_{-\ell}\dots \vartheta_{\ell}]\times J$ be an open set such that $(\vartheta,y)\in V$.
Since $F(\Omega)=\Omega$ we get $F^{-\ell}(\vartheta,y)\cap \Omega\neq \emptyset$. This
implies 
that 
$$
(f_{\vartheta_{-1}}\circ \cdots\circ f_{\vartheta_{-\ell}})^{-1}(y)\cap \overline{A_{F}}\neq \emptyset
\quad \Longrightarrow \quad
A_{F}\cap f_{\vartheta_{-\ell}}^{-1}\circ \cdots\circ f_{\vartheta_{-1}}^{-1}(J)\neq \emptyset.
$$
Applying Lemma \ref{l.interesecao} to the open set $Q=\Sigma_k^+$ we get a sequence $\xi^{+}\in \Sigma_k^+$
and $n\ge 0$ such that
  $$
F^{n}\big( [\xi^{+}]\times M \big)
\subset  [-\ell;\vartheta_{-\ell}\dots \vartheta_{\ell}]\times J. 
$$ 
Note that by the definition of $S_F$ in \eqref{e.stmenos}, we have that if $\vartheta= \vartheta^-.\vartheta^+\in S_F$
then $\vartheta^-.\xi^+\in S_F$ for all $\xi^+ \in \Sigma_k^+$.
Since $S_F\ne \emptyset$ by hypothesis, this implies that
$ \big( [\xi^{+}]\times M \big)\cap \mathrm{graph}\,\rho\neq \emptyset$
 and that $\mathrm{graph}\,\rho$ is $F$-invariant, 
we conclude that 
$$
 \mathrm{graph}\,\rho\cap [-\ell;\vartheta_{-\ell}\dots \vartheta_{\ell}]\times J\neq \emptyset.
 $$
Since this holds for every neighborhood $V$ of  $(\vartheta,x)$ we get $(\vartheta,x)\in \overline{\mathrm{graph}\,\rho}$. As this holds for every point of $\Omega$ the lemma follows.
\end{proof}

We are now ready to prove item (2).
Recall that  $F_{A}$ is the restriction of $F$ to the set  $\Sigma_{k}\times \overline{A_{F}}$.
Clearly,  $\Lambda_{F_{A}}\subset\Sigma_{k}\times \overline{A_{F}}$
and since 
$\Lambda_{F_{A}}$ is a maximal invariant set we have that 
$F(\Lambda_{F_{A}})= F_{A}(\Lambda_{F_{A}})=\Lambda_{F_{A}}$.
Thus we can apply Lemma \ref{l.invariancep} to $\Lambda_{F_{A}}$ to obtain
$\Lambda_{F_{A}}\subset \overline{\mathrm{graph}\,\rho}$. 
For the converse  inclusion  ``$\supset$'', it is enough to
take any disjunctive point $z=(\vartheta,x)\in \Lambda_{F_{A}}$,
then the first item of the theorem 
and the $F$-invariance of  $\Lambda_{F_{A}}$
imply that
$$
 \overline{\mathrm{graph}\,\rho}\subset \omega_{F}(z)\subset \Lambda_{F_{A}},
$$
proving the inclusion.

Item (3) of the Theorem($F$ is topologically mixing
in  $\overline{\mathrm{graph}\,\rho}$)
is an immediate consequence of  the next lemma.

\begin{lema}\label{l.topologicallymixing}
Consider any pair of open sets 
$$
V=[-\ell;b_{-\ell}\dots b_{\ell}]\times J
\quad
\mbox{and}
\quad
U=[-m;a_{-m}\dots a_{m}]\times I
$$ 
intersecting 
$\overline{\mathrm{graph}\,\rho}$. 
Then 
there is $n_0$ such that
$$
F^{n}(U\cap \overline{\mathrm{graph}\,\rho})\cap V
\ne\emptyset
\quad
\mbox{for every $n\ge n_0$.}
$$
\end{lema}

\begin{proof}
We first state a claim that will be also used to conclude the proof of item (1).
%Note that the \textcolor{red}{first part of the claim} implies 
%\textcolor{red}{the first part of
%item (2).}

\begin{claim}\label{cl.doisemum}
%$F(\overline{\mathrm{graph}\,\varrho})=\overline{\mathrm{graph}\,\varrho}$.  
Let $(\vartheta,q)\in\overline{\mathrm{graph}\,\rho}$. Then
$
\overline{A_{F}}\cap f_{\vartheta_{-\ell}}^{-1}\circ \cdots\circ f_{\vartheta_{-1}}^{-1}(q)\neq \emptyset
$
for every $\ell>0$.
\end{claim}
\begin{proof} 
Take any $(\vartheta,q)\in\overline{\mathrm{graph}\,\rho}$.
 Since by item (2)  $\Lambda_{F_{A}}=\overline{\mathrm{graph}\,\rho}$,
 we have that $F(\overline{\mathrm{graph}\,\rho})=\overline{\mathrm{graph}\,\rho}$. Thus
 $F^{-\ell}(\vartheta,q)\cap \overline{\mathrm{graph}\,\rho} \neq \emptyset$ for every $\ell>0$.
 By definition, we have that 
$\overline{\mathrm{graph}\,\rho}\subset\Sigma_k \times\overline{A_{F}}$ which 
implies $(f_{\vartheta_{-1}}\circ \cdots\circ f_{\vartheta_{-\ell}})^{-1}(q)\cap \overline{A_{F}}\neq \emptyset$, proving the claim.
\end{proof}

By hypothesis,
the set  $V=[-\ell;b_{-\ell}\dots b_{\ell}]\times J$ 
contains a point $(\vartheta,q)\in\overline{\mathrm{graph}\,\rho}$. Hence,
by Claim~\ref{cl.doisemum},
$$
(f_{\vartheta_{-1}}\circ \cdots\circ f_{\vartheta_{-\ell}})^{-1}(q)\cap \overline{A_{F}}\neq \emptyset
\quad \implies \quad
A_{F}\cap f_{\vartheta_{-\ell}}^{-1}\circ \cdots\circ f_{\vartheta_{-1}}^{-1}(J)\neq \emptyset.
$$
 This allows to apply Lemma~\ref{l.interesecao} to the sets $V=[-\ell;b_{-\ell}\dots b_{\ell}]\times J$
 and $Q=[0; a_{0}\dots a_{m}]\subset \Sigma_k^+$ (where $a_0\dots a_m$ is an the definition of $U$), getting
  $m_{0}=m_0(Q)$ such that for every $n\geq m_{0}$ 
 there is a sequence $\xi^{+,n}\in Q$ such that 
 \begin{equation}
 \label{e.warsawlastday}
F^{n}\big( [\xi^{+,n}] \times M \big)
\subset  [-\ell;b_{-\ell}\dots b_{\ell}]\times J=V \quad \mbox{for every}\quad n\geq m_{0}.
\end{equation}

\begin{claim}
\label{cl.semnome}
$\big( \mathrm{graph}\,\rho \cap U \big) \cap \big([\xi^{+,n}]\times M\big)\neq \emptyset $ 
for every $ n\geq \max\{m, m_{0}\}$.
\end{claim}

\begin{proof}
By hypothesis, $U=[-m;a_{-m}\dots a_{m}]\times I$ and
there is 
$(\zeta,z)\in U\cap {\mathrm{graph}\,\rho}$
where
$\zeta\in S_{F}$. 
Consider the sequence $\gamma= \zeta^-.\xi^{+,n}\in
\Sigma_k$.    By construction, 
$$
\gamma \in S_{F} \cap [\xi^{+,n}] \cap [-m;a_{-m}\dots a_{m}]
$$
and $\rho (\gamma)=z$. Hence $(\gamma, \rho(\gamma))$ is in the intersection set 
in the claim.
\end{proof}

Using Claim~\ref{cl.semnome} and \eqref{e.warsawlastday}
 we get
$$
 F^{n}(U\cap \mathrm{graph}\,\rho)\cap V\neq \emptyset\quad \mbox{for all}\quad  n\geq n_0=\max\{m, m_{0}\},
 $$
proving the lemma.
 \end{proof}

It remains to prove the second part of   item (1). By the first part of 
 item (1) it remains to see that  (under  appropriate assumptions)  fixed any disjunctive point
 $z=(\vartheta,x)$  one has
 $\omega_{F}(z)\subset\overline{\mathrm{graph}\,\rho}$.

 First
   suppose that $\overline{A_{F}}$
 has nonempty interior.  
 This implies that there is $\xi_{-1}\dots \xi_{-r}$ such that
 $f_{\xi_{-1}} \circ \cdots \circ f_{\xi_{-r}}  (M)$ is contained in the interior of $\overline{A_{F}}$.
 Since 
 %the interior of $\overline{A_{\mathrm{t}}}$ is not empty 
 %and 
 $\vartheta$ is disjunctive 
 there is a large $n_{0}$ such that $\vartheta_{n_0-r}\dots \vartheta_{n_0-1} \dots =
 \xi_{-r}\dots \xi_{-1}$. Hence
$\bar z=F^{n_{0}}(z)\in \Sigma_{k}\times\overline{A_{F}}$.
 By the invariance of $\overline{A_{F}}$ it follows that $F^{n}(z)\in \Sigma_{k}\times\overline{A_{F}}$
for every $n\geq n_{0}$ and hence
 $\omega_{F}(\bar z)\subset\Sigma_{k}\times\overline{A_{F}}$.
  Since $\omega_{F}(\bar z)=\omega_{F}(z)$ and $\omega_{F}(z)$ is $F$-invariant, 
  Lemma \ref{l.invariancep} implies that 
   $\omega_{F}(z)\subset\overline{\mathrm{graph}\,\rho}$.

Suppose now  that $\overline{A_{F}}$ is Lyapunov stable.
%Take any point $z=(\xi,x)\in \mathcal{D}\times X$.
%To prove the inclusion
%$\omega(z)\subset\overline{\mathrm{graph}\,\varrho}$ 
Take any
 open neighbourhood  $U$ of $\overline{A_{F}}$ and note that
there is a neighbourhood $V$ of 
 $\overline{A_{F}}$ such that $\mathcal{B}_{F}^{n}(V)\subset U$ 
 for every $n\geq 0$. 
 The definition of $\overline{A_{F}}$
 and the fact the $\vartheta$ is disjunctive imply that there is 
 $n_{0}$ such that $F^{n_{0}}(\vartheta,x)\in \Sigma_{k}\times V$. Hence
 $F^{n}(\vartheta,x)\in \Sigma_{k}\times U$ for every $n\geq n_{0}$, which implies 
 that $\omega_{F}(z)\subset \Sigma_{k}\times U$. Since $U$ is an arbitrary 
 neighbourhood of $\overline{A_{F}}$ we have that $\omega_{F}(z)\subset \Sigma_{k}\times \overline{A_{F}}$.
 It follows from Lemma \ref{l.invariancep} that 
 $\omega_{F}(z)\subset \overline{\mathrm{graph}\,\rho}$. The proof of the theorem is now complete.
 \hfill \qed

\section{Invariant sets with disconnected spines. Proof of Theorem \ref{mt.espinhodesconexo}}
\label{s.espinhodesconexo}

\subsection{Auxiliary $\mathcal{K}$-pairs}\label{ss.kpairs}
To construct the open set $\mathcal{S}_m$ in Theorem~\ref{mt.espinhodesconexo}
we need to recall the construction by Kudryashov in \cite{Yuri}.
%Let us recall the construction  
%of the bony attractors in \cite{Yuri}.

\begin{defi}[$\mathcal{K}$-pairs]
\label{d.kpair}
{\em{
Consider $C^1$ increasing maps
 $f_{1},f_{2}\colon [c,d]\to [c,d]$. Let $J=[a,b]\subset[c,d]$ 
be such that $a$ is a globally attracting fixed point of $f_{1|J}$ and $b$ is a globally attracting fixed point
 of $f_{2|J}$.
We say that $(f_{1|J},f_{2|J})$ is a \emph{$\mathcal{K}$-pair} for $J$ if for every small $\epsilon>0$ there are
$C^1$ neighborhoods 
$\mathcal{V}_1$ and $\mathcal{V}_2$ of $f_1$ and $f_2$ respectively
such that for every pair of maps 
$g_1 \in \mathcal{V}_1$ and $g_2 \in\mathcal{V}_2$
 the following holds:    
\begin{enumerate}
\item 
Then there are points $p_{1}<p_{2}$ such that $|p_{1}-a|<\epsilon$, $|p_{2}-b|<\epsilon$ and 
$p_{i}$ is 
 a globally attracting fixed point of $g_{i|J'}$, $i=1,2$, where $J'=[p_{1},p_{2}]$.
\item
$g_{1|J'}\circ g_{2|J'}$ has a repelling fixed point 
\item 
There are uniform contractions $h_{1}\dots,h_{\ell}\in \mathrm{IFS}(g_{1|J'},g_{2|J'})$ such that 
$$
J'=\bigcup_{i}^{\ell} h_{i}(J').
$$
\end{enumerate}}}
\end{defi}

%\textcolor{magenta}{
%Consider a $C^{1}$ step skew product $F$ with}

%
%
%
% $f_{1}$ is the piecewise-linear map
%with ``vertices'' $(0, 0)$, $(0.6, 0.2)$, and $(1, 0.8)$
%and $T_{2}$ is the piecewise-linear map with ``vertices''
%$(0, 0.15)$, $(0.4, 0.8)$, and $(1, 1)$, see Figure~\ref{bony}.
%In \cite{Yuri} it is proved that 
%\begin{enumerate}
%\item
%$f_{1}\circ f_{2}$ has a repelling fixed point, 
%\item
%the compositions $f_{1}^3$, $f_{1}^{2}\circ f_{2}$, $f_{2}^{2}\circ f_{1}$ and $f_{2}^{5}$ are uniform contractions,
%\item
%$f_{1}^{2}\circ f_{2}$, $f_{2}^{2}\circ f_{1}([0,1]) \cup  
%f_{2}^{5}([0,1])=[0,1]$.
%\end{enumerate}

\begin{figure}[h!]
\centering
\begin{tikzpicture}[xscale=3,yscale=3]
 \node[scale=0.6,left] at (0.3,0.75) {$f_{2}$};
  \node[scale=0.6,below] at (0.7,0.6) {$f_{1}$};
\draw[->] (0,0)-- (1,0);
\draw[->] (0,0)--(0,1);
%\draw[red,-]  (0,0)--(0.6,0.2);
%\draw[red,-]  (0.6,0.2)--(1,0.75);
%\draw[blue,-]  (0,0.15)--(0.4,0.8);
%\draw[blue,-]  (0.4,0.8)--(1,1);
\draw[dashed] (0,1)--(1,1);
\draw[dashed] (0,0)--(1,1);
\draw[dashed] (1,0)--(1,1);
 \draw[red, scale=1] (0,0.15) sin (1,1); 
   %  \draw[magenta, scale=1] (0,0) cos (0.7,0);
     \draw[scale=1.2,domain=0:0.83,smooth,variable=\x,blue] plot ({\x},{\x*\x});
\end{tikzpicture}
\caption{A $\mathcal{K}$-pair.}
\label{bony}
\end{figure}
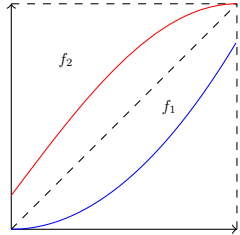

\begin{remark}[Existence of $\mathcal{K}$-pairs and minimality] \label{r.kpairs}
{\emph{
The existence of $\mathcal{K}$-pairs on 
a prescribed interval $[a,b]$ is provided by Kudryashov in 
\cite{Yuri}.
Note also that the $\mathcal{K}$-pairs form
a $C^1$-open family.}}

{\em{
Using the notation above, if
 $(f_{1|J},f_{2|J})$ is a \emph{$\mathcal{K}$-pair} for $J$, then
  $\mathfrak{G}=\mathrm{IFS}(g_{1|J'},g_{2|J'})$ is {\emph{minimal,}} that is,
 given any point $x\in J'$ its $\mathfrak{G}$-orbit  
$$
\mathfrak{G}(x)\eqdef\{g(x)\colon g\in \mathrm{IFS}(g_{1|J'},g_{2|J'})\}
$$ 
is dense in $J'$ (here $J'=[p_1,p_2]$ as above). To see why this is so, let $\mathfrak{K}=\mathrm{IFS}(h_{1},\dots,h_{\ell})$, where $h_1,\dots, h_\ell$ are as in item (3). It follows from 
\cite{Hu} that the target set $A_{\mathfrak{K}}$ (that is, the target set 
of the skew product whose fiber maps are $h_{1},\dots,h_{\ell}$) coincides with $J'$. Since $\mathfrak{K}$ is a subsystem of 
$\mathfrak{G}$ then we have $A_{\mathfrak{K}}\subset A_{\mathfrak{G}}\subset J'$ which implies that 
$A_{\mathfrak{G}}= J'$. It is clear that $A_{\mathfrak{G}}= J'$ implies minimality.}}
\end{remark}

%$A_{F}=J$. For the second claim first note that
% $\mathrm{IFS}(g_{1},\dots, g_{\ell})$ is 
% hyperbolic
%(contracting) and that $J$ is the unique fixed point of its Barnsley-Hutchinson operator. This implies that $A_{G}=J$ where $G\colon \Sigma_{\ell}\times [0,1]\to \Sigma_{\ell}\times [0,1] $ is the step skew product whose fibers maps are $g_{1},\dots, g_{\ell}$.
% To see that $A_{F}=J$  just note that $A_{G}\subset A_{F}$ and
% 
% 
%Finally we observe that the property $A_{F}=J$ implies that $\mathfrak{G}=\mathrm{IFS}(f_{1|J},f_{2|J})$ is minimal, 
%that is,

%
%
%
%Observe that any pair 
%Let $F\colon\Sigma_{2}\times [0,1]\to \Sigma_{2}\times [0,1]$ be a step skew product over the shift
% whose fiber maps are $f_{1}$ and $f_{2}$.
%A pair $(f_{1},f_{2})$ is called a $\mathcal{K}$-pair if it has an open neighbourhood $\mathcal{V}$ such that 
%every pair

%
%  satisfying condition (1),(2) and (3) is called a $\mathcal{K}$-pair.
%In \cite{Yuri} is presented a open set of $C^{1}$ skew products whose fiber maps consist of a $\mathcal{K}$-pair.
% 
%
%Note that in this case
%$S_F\ne \Sigma_k$ 
%(just note that the periodic sequence $\overline{12}$ does not belong to $S_{F}$)
%and that

\subsection{Proof of Theorem~ \ref{mt.espinhodesconexo}}
Fix $m\ge 1$ and consider $m$ disjoint subintervals of $[0,1]$, say
 $I_{1}=[a_{1},b_{1}],\dots, I_{m}=[a_{m},b_{m}]$,
 with $a_{1}=0<b_{1}< b_{2}<\dots <a_m < b_{m}=1$
 and four strictly increasing
    $C^1$ maps  $f_1,f_2,f_3, f_4\colon [0,1]\to [0,1]$ such that (see Figure ~\ref{f.4intervalos}):
  \begin{itemize}  
  \item[(a)]
  $f_{1}(I_{i})\subset I_{i}$, $f_{2}(I_{i})\subset I_{i}$ and 
 $(f_{1|I_{i}}, f_{2|I_{i}})$ is a $\mathcal{K}$-pair for $I_i$, for every  $i\in \{1,\dots,m\}$,
\item[(b)]
  $f_3$ is a contraction with $f_3([0,1])\subset I_1$, and
  \item[(c)] 
$f_4(I_i)\subset \textrm{int}(I_{i+1})$ for $i\in \{1,\dots, m-1\}$
and $f_4(I_m)\subset \textrm{int}(I_{m})$.
\end{itemize}

\begin{figure}[!h]
  \centering
  \begin{tikzpicture}[xscale=3.5,yscale=3.5]
 
     \draw[->] (0,0) -- (1.5,0);
    \draw[->] (0,0) -- (0,1.5);
   
    \draw[gray,dashed] (0,0.25)--(0.25,0.25);
     \draw[gray,dashed] (0.25,0)--(0.25,0.25);
     \draw[gray,dashed] (0.4,0)--(0.4,0.25);
      \draw[gray,dashed] (0.4,0.25)--(0.65,0.25);
     \draw[gray,dashed] (0.65,0)--(0.65,0.25);
      \draw[gray,dashed] (0.75,0)--(0.75,0.25);
      \draw[gray,dashed] (0.75,0.25)--(1,0.25);
     \draw[gray,dashed] (1,0)--(1,0.25);
     
     \draw[gray,dashed] (0,0.4)--(0.25,0.4);
     \draw[gray,dashed] (0.25,0.4)--(0.25,0.65);
     \draw[gray,dashed]  (0,0.65)--(0.25,0.65);
     
      \draw[gray,dashed] (0,0.75)--(0.25,0.75);
     \draw[gray,dashed] (0.25,0.75)--(0.25,1);
     \draw[gray,dashed]  (0,1)--(0.25,1);
     
     \draw[gray,dashed]  (0,1.25)--(0.25,1.25);
      
     \draw[gray,dashed]  (0,1.5)--(0.25,1.5);
    \draw[gray,dashed]    (0.25,1.25)--(0.25,1.5);

     \draw[gray,dashed] (0.4,0.4)--(0.4,0.65);
      \draw[gray,dashed] (0.4,0.65)--(0.65,0.65);
      \draw[gray,dashed] (0.4,0.4)--(0.65,0.4);
      \draw[gray,dashed] (0.65,0.4)--(0.65,0.65);

     \draw[gray,dashed] (0.4,0.75)--(0.4,1);
      \draw[gray,dashed] (0.4,1)--(0.65,1);
      \draw[gray,dashed] (0.4,0.75)--(0.65,0.75);
      \draw[gray,dashed] (0.65,0.75)--(0.65,1);
      
      \draw[gray,dashed]  (0.4,1.25)--(0.4,1.5);
      \draw[gray,dashed]  (0.4,1.5)--(0.65,1.5);
      \draw[gray,dashed]  (0.65,1.5)--(0.65,1.25);
      \draw[gray,dashed]  (0.4,1.25)--(0.65,1.25);

     \draw[gray,dashed] (0.75,0.4)--(0.75,0.65);
      \draw[gray,dashed] (0.75,0.65)--(1,0.65);
      \draw[gray,dashed] (0.75,0.4)--(1,0.4);
      \draw[gray,dashed] (1,0.4)--(1,0.65);  
     \draw[gray,dashed] (0.75,1.25)--(0.75,1.5); 
     \draw[gray,dashed] (0.75,1.25)--(1,1.25); 
     \draw[gray,dashed] (0.75,1.5)--(1,1.5);
      \draw[gray,dashed] (1,1.25)--(1,1.5);
      
      % \draw[gray,dashed] (0.75,1.25)--(1,1.25);

     \draw[gray,dashed] (0.75,0.75)--(0.75,1);
      \draw[gray,dashed] (0.75,1)--(1,1);
      \draw[gray,dashed] (0.75,0.75)--(1,0.75);
      \draw[gray,dashed] (1,0.75)--(1,1);
      
      \draw[gray,dashed] (1.25,1.25)--(1.25,1.5);
      \draw[gray,dashed] (1.25,1.25)--(1.5,1.25);
      \draw[gray,dashed] (1.25,1.5)--(1.5,1.5);
      \draw[gray,dashed] (1.5,1.25)--(1.5,1.5);

      \draw[gray,dashed] (1.25,0.75)--(1.25,1);
      \draw[gray,dashed] (1.25,0.75)--(1.5,0.75);
      \draw[gray,dashed] (1.25,1)--(1.5,1);
      \draw[gray,dashed] (1.5,0.75)--(1.5,1);

      \draw[gray,dashed] (1.25,0.4)--(1.25,0.65);
      \draw[gray,dashed] (1.25,0.4)--(1.5,0.4);
      \draw[gray,dashed] (1.25,0.65)--(1.5,0.65);
      \draw[gray,dashed] (1.5,0.4)--(1.5,0.65);

      \draw[gray,dashed] (1.25,0)--(1.25,0.25);
      \draw[gray,dashed] (1.25,0.25)--(1.5,0.25);
      \draw[gray,dashed] (1.5,0)--(1.5,0.25);

      \draw[gray,dashed] (0,0)--(1.5,1.5);

       \draw[blue,-] (0,0)--(0.15,0.05);
      \draw[blue,-] (0.15,0.05)--(0.25,0.18);
       \draw[blue,-] (0.25,0.18)--(0.3,0.35);
         \draw[blue,-](0.3,0.35)--(0.4,0.4);
        % \draw[blue,-](0.4,0.77)--(0.75,0.787);      
          \draw[blue,-] (0.4,0.4)-- (0.55,0.45);
            \draw[blue,-] (0.55,0.45)-- (0.65,0.6);
        \draw[blue,-] (0.75,0.75)--(0.9,0.8);
      \draw[blue,-] (0.9,0.8)--(1,0.93);
       \draw[blue,-] (0.65,0.6)--(0.69,0.72);
        \draw[blue,-] (0.69,0.72)--(0.75,0.75);
        \draw[blue,-](1.25,1.25)--(1.4,1.3);
        \draw[blue,-](1.4,1.3)--(1.5,1.43);
        \draw[blue,-] (1,0.93)--(1.10,1.15);
        \draw[blue,-](1.10,1.15)--(1.25,1.25);
       
     \draw[->](-0.75,0.75)--(-0.5,0.75);
   \draw[->](-0.75,0.75)--(-0.75,1);   
      \draw[gray,dashed] (-0.75,1)--(-0.5,1);
        \draw[gray,dashed]  (-0.5,0.75)--(-0.5,1);
      \draw[red,-] (-0.75,0.787)--(-0.65,0.95);
      \draw[red,-]  (-0.65,0.95)--(-0.5,1);
       \draw[blue,-] (-0.75,0.75)--(-0.6,0.8);
      \draw[blue,-]  (-0.6,0.8)--(-0.5,0.93);

       \draw[red,-]  (0,0.037)--(0.1,0.2);
        \draw[red,-]  (0.1,0.2)--(0.34,0.279);
        \draw[red,-] (0.34,0.279)--(0.4,0.437);   
     \draw[red,-](0.65,0.65)--(0.72,0.67);   
         \draw[red,-] (0.72,0.67)--(0.75,0.787); 
        \draw[red,-](0.4,0.437)--(0.5,0.6);
         \draw[red,-](0.5,0.6)--(0.65,0.65);
         \draw[red,-]  (0.75,0.787)--(0.85,0.95);
        \draw[red,-]  (0.85,0.95)--(1,1);
       \draw[red,-] (1.25,1.287)--(1.35,1.45);
       \draw[red,-](1.35,1.45)--(1.5,1.5);
        \draw[red,-](1,1)--(1.15,1.10);
        \draw[red,-](1.15,1.10)--(1.25,1.287);
        
\draw[magenta,-]   (0,0.5)--(0.3,0.6);

\draw[magenta,-]    (0.3,0.6)--(0.333,0.8);

\draw[magenta,-]    (0.333,0.8)--(0.68,0.9);      
\draw[magenta,-]     (0.68,0.9)--(0.72,1.3);
\draw[magenta,-]     (0.72,1.3)--(1.5,1.4);

        \draw[-](0,0.11)--(1.5,0.2);
      
       \node[scale=0.8,left] at (0,0.15) {$I_{1}$};
        \node[scale=0.8,left] at (0,0.5) {$I_{2}$};
        \node[scale=0.8,left] at (0,0.9) {$I_{3}$};
        \node[scale=0.8,left] at (0,1.35) {$I_{4}$};        
        
         \node[scale=0.8,below] at (0.15,0) {$I_{1}$};
        \node[scale=0.8,below] at (0.5,0) {$I_{2}$};
        \node[scale=0.8,below] at (0.9,0) {$I_{3}$};
         \node[scale=0.8,below] at (1.35,0) {$I_{4}$};
         
          \node[scale=0.5] at (-0.38,0.9) {$\mathcal{K}$-pair};
         \draw[gray,dashed,->] (-0.28,0.9)--(0,0.15);  
         \draw[gray,dashed,->] (-0.28,0.9)--(0.4,0.5);
         \draw[gray,dashed,->]  (-0.28,0.9)--(0.75,0.9);
          \draw[gray,dashed,->]  (-0.28,0.9)--(1.25,1.35);

    %\draw[green,smooth,samples=100,domain=1.0:2.0] plot(\x,{\x});
% 
%       \draw[dashed] (2,0) -- (2,2);
%       \draw[dashed] (0,2) -- (2,2);
%        \draw[dashed] (0,1) -- (1,1); 
%         \draw[dashed] (1,0) -- (1,1);
%       \draw[green,smooth,samples=100,domain=0.0:1.0] plot(\x,{0.33333*\x+0.66666});
%     \draw[-] (0,0) -- (2,0);
%      \draw[-] (0,0) -- (0,2);
%      \draw[blue,smooth,samples=100,domain=0.0:2.0] plot(\x,{0.33333*\x});
%       \node[scale=1,left] at (0,0.66666) {$\frac{2}{3}$};
       \end{tikzpicture}
     \caption{A $\mathcal{K}$-pair and a piecewise linear model.}\label{f.4intervalos}
 
     \end{figure}
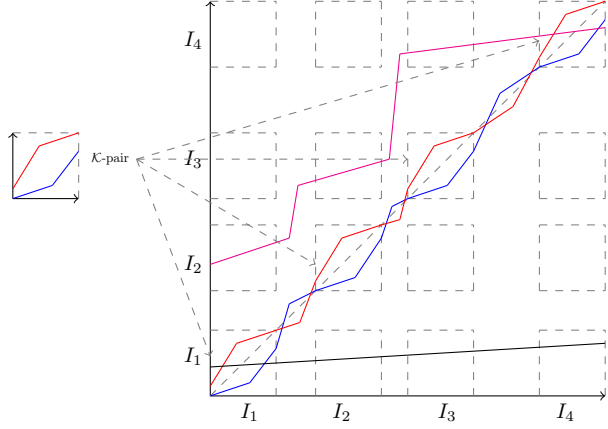

This construction can be done by using $C^{1}$ maps 
in a robust way (recall Remark~\ref{r.kpairs}). 
\begin{prop}\label{p.bones}
Let $F\colon \Sigma_{4}\times [0,1]
\to \Sigma_{4}\times [0,1]$ be a step skew product  whose fiber maps
$f_{1},f_{2},f_{3},$ and $f_{4}$
satisfy properties (a), (b), and (c). 
 Then
 \begin{itemize}
 \item[(1)] For every disjunctive point $z$  it holds 
$\overline{\mathrm{graph}\,\rho}=\omega_F(z)$.
 \item[(2)]
 There is a dense subset 
 $\Gamma$
 of $\Sigma_{4}$  such that
 for every sequence $\vartheta\in \Gamma$ the intersection of the fibre $\{\vartheta\}\times [0,1]$ with 
 $\overline{\mathrm{graph}\,\rho}$ is the
 union of $m$ disjoint nontrivial intervals. 
  \end{itemize}
 \end{prop}

Theorem~\ref{mt.espinhodesconexo} follows from this proposition and Remark~\ref{r.kpairs}: note that if $F\colon \Sigma_k \times [0,1]\to \Sigma_k \times [0,1]$ is such that its fiber maps have a $\mathcal{K}$-pair the same holds for every $G\colon \Sigma_k \times [0,1]\to \Sigma_k \times [0,1]$ sufficiently $C^1$ close to $F$. Also note that the other properties of the fiber maps of $F$ hold for small $C^1$ perturbations.

\begin{proof}[Proof of Proposition~\ref{p.bones}]
The next lemma implies  that $\overline{A_{F}}$ has a nonempty interior, thus
the first item of the proposition follows from
the second part of item (1) in Theorem~\ref{mt.mixing}.

\begin{lema}\label{l.overlinetarget}
$
\overline{A_{F}}=\bigcup_{i=1}^{m} I_{i}.
$
\end{lema}

\begin{proof}
Observe first that, by construction, the set $\bigcup_{i=1}^{m} I_{i}$
is invariant under the Barnsley-Hutchinson operator $\mathcal{B}_F$ induced by $F$, 
recall \eqref{e.BH} for the definition. 

To prove the inclusion  ``$\subset$'' we argue by contradiction,
 suppose that there is $x\in \overline{A_{F}}$ that 
does not belong in $\bigcup_{i=1}^{m} I_{i}$.  Since this union is compact, 
there is an open neighborhood $V$ of $x$ such that $ V\cap \bigcup_{i=1}^{m} I_{i}=\emptyset$. 
By definition of the target set there is a composition $f_{a_{1}}\circ\dots \circ f_{a_{\ell}}$ such that 
$$
f_{a_{1}}\circ\dots \circ f_{a_{\ell}}([0,1])\subset V.
$$
This contradicts the invariance of  $\bigcup_{i=1}^{m} I_{i}$ and  proves the inclusion ``$\subset$''.

We now see  that $I_{i}\subset \overline{A_{F}}$ for every $i$. We  first 
claim that $I_i\cap  \overline{A_{F}}\ne \emptyset$.
For that consider the fixed point $q$ of the contraction $f_3$ and note that $q\in \overline{A_{F}}$ and
that, 
by construction, $q_i\eqdef  f_4^{i-1}(q)\in I_i$ for every  for $i\in \{2,\cdots,m\}$.
%\margem{acho que necessitamos $f_4^j(q_1)$ in $I_m$ for all $j\ge m$. verificar , 
%$m=3$ $f^{j}(q)\in I_{3}$ for every $j\geq 3$ is false, o $q$ esta mudando de intervalo em cada iteracao}
The 
$\mathcal{B}_{F}$-invariance of
 $\overline{A_{F}}$ implies that
 $I_i\cap \overline{A_{F}}\ne \emptyset$.
  
By the minimality of $\mathfrak{G_i}=\mathrm{IFS}(f_{1|I_{i}}, f_{2|I_{i}})$ 
(recall Remark~\ref{r.kpairs})
we have 
 that the $\mathfrak{G_{i}}$-orbit of each $p_i$ is dense  in $I_{i}$.
 Since $\overline{A_{F}}$ is $\mathcal{B}_{F}$-invariant it follows that
 $I_i\subset \overline{A_{F}}$. The proof of the lemma is now complete.
\end{proof}

%
%\textcolor{blue}{To prove item (2) we first see  that there is a sequence $\vartheta$ such that 
%$(\{\vartheta\}\times [0,1])\cap  \overline{\mathrm{graph}\,\rho}$ 
%is the union of $m$ disjoint non-trivial intervals.}

Denote by $\Sigma_4 (\overline{12}\,^-)$ the subset of $\Sigma_4$ consisting of the sequences  $\eta=\eta^-.
\eta^+$ with $\eta^-=12121212\dots$.

\begin{lema}
\label{l.mdisjoint}
Let  $\vartheta \in  \Sigma_k (\overline{12}\,^-)$.
Then
$(\{\vartheta\}\times [0,1])\cap  \overline{\mathrm{graph}\,\rho}$ 
is the union of $m$ disjoint nontrivial intervals.
\end{lema}

\begin{proof}
From item (2) in Theorem \ref{mt.mixing} it follows that  $\overline{\mathrm{graph}\,\rho}=\Lambda_{F_A}$
and 
%is the maximal attractor of the map
 %$F_{A}\eqdef F_{|\Sigma_{4}\times \overline{A_{F}}}$. 
 by the characterization
of $\Lambda_{F_A}$
in Proposition \ref{p.maximalattractor} it follows  that 
$$
 \overline{\mathrm{graph}\,\rho}=\bigcup_{\vartheta\in \Sigma_k} \{\vartheta\}\times I^{\mathrm{t}}_{\vartheta},
\quad
\mbox{where} 
\quad
I^{\mathrm{t}}_{\vartheta}=\bigcap_{n\geq 1} f_{\vartheta_{-1}}\circ\cdots\circ f_{\vartheta_{-n}}(\overline{A_{F}}). 
$$
Therefore 
$$
(\{\vartheta\}\times [0,1])\cap  \overline{\mathrm{graph}\,\rho}=
\{\vartheta\}\times I^{\mathrm{t}}_{\vartheta}
\quad
\mbox{for every $\vartheta\in \Sigma_{k}$.}
$$ 
%Hence to prove the lemma it is enough to  see that 
%$I^{\mathrm{t}}_{\vartheta}$ is a union of $m$ disjoint intervals for every  $\vartheta$
%with $\vartheta_{-1}\dots \vartheta_{-n}\ldots=12121212\dots$.
%For that 
% it is enough to consider any sequence  $\vartheta$
% such that $\vartheta_{-2n+1}=1$
%  and $\vartheta_{-2n}=2$ for all $n\geq 1$. To see that $I^{\mathrm{t}}_{\vartheta}$ 
%  disjoint intervals 
Recall that by Lemma~\ref{l.overlinetarget} we have $\overline{A_{F}}= \bigcup_{i=1}^{m} I_{i}$, which implies that for every $\vartheta$
 $$
I^{\mathrm{t}}_{\vartheta}= \bigcap_{n\geq 1} f_{\vartheta_{-1}}\circ\cdots\circ f_{\vartheta_{-n}}
(\overline{A_{F}})=
\bigcap_{n\geq 1} \bigcup_{i=1}^{m}  f_{\vartheta_{-1}}\circ\cdots\circ f_{\vartheta_{-n}}(I_{i}).
$$

Let $\vartheta=\vartheta^-.\vartheta^+\in \Sigma_4$ be a sequence 
such that $\vartheta^-$ consists of $1$ and $2$ and
$I_{\vartheta}^{\mathrm{t}}(i)$ be the spine of $\vartheta$ with respect to the $\mathrm{IFS}(f_{1|I_{i}}, f_{2|I_{i}})$, that is
$$
I_{\vartheta}^{\mathrm{t}}(i)\eqdef
 \bigcap_{n\geq 1}  f_{\vartheta_{-1}}\circ\cdots\circ f_{\vartheta_{-n}}(I_{i}).
$$

\begin{claim}\label{c.restriction} 
For every $i\in \{1,\dots, m\}$ and every sequence $\vartheta=\vartheta^-.\vartheta^+\in \Sigma_4$ such that $\vartheta^-$ consists of $1$ and $2$  it holds
$$
I^{\mathrm{t}}_{\vartheta}=\bigcup_{i=1}^{m}I_{\vartheta}^{\mathrm{t}}(i).
$$
 \end{claim}
 
\begin{proof}
Recalling the definition of $I_{\vartheta}^{\mathrm{t}}(i)$ we have that 
it is enough to
see that 
$$
\bigcap_{n\geq 1} \bigcup_{i=1}^{m}  f_{\vartheta_{-1}}\circ\cdots\circ f_{\vartheta_{-n}}(I_{i})=
\bigcup_{i=1}^{m} \bigcap_{n\geq 1}  f_{\vartheta_{-1}}\circ\cdots\circ f_{\vartheta_{-n}}(I_{i}).
$$
The inclusion $``\supset  "$ it is straightforward.
To get the inclusion $``\subset "$ take a point $p\in \bigcap_{n\geq 1} \bigcup_{i=1}^{m}  f_{\vartheta_{-1}}\circ\cdots\circ f_{\vartheta_{-n}}(I_{i})$ and note that for every $n\geq 1$
there is $i_{n}$ such that $p\in f_{\vartheta_{-1}}\circ\cdots\circ f_{\vartheta_{-n}}(I_{i_{n}})$. 
Since $f_{j}(I_{i})\subset I_{i}$ for every $j=1,2$ and every $i\in \{1,\dots,m\}$ and 
 the intervals $I_{i}$'s are pairwise disjoint we conclude that $i_{n}$ is independent of $n$,
say $i_n=s$ for all $n\geq 1$.
Therefore
$$
p\in f_{\vartheta_{-1}}\circ\cdots\circ f_{\vartheta_{-n}}(I_{s}) \quad \mbox{for every}\quad n\geq 1
\quad \implies \quad
p\in \bigcup_{i=1}^{m} \bigcap_{n\geq 1}  f_{\vartheta_{-1}}\circ\cdots\circ f_{\vartheta_{-n}}(I_{i}),
$$
proving the claim.
\end{proof}

Note that the fact that  $f_{1|I_{i}}\circ f_{2|I_{i}}$ has a repelling fixed point for  every $i\in \{1,\dots, m\}$
implies that
 $I_{\xi}^{\mathrm{t}}(i)$ is a nontrivial interval for every $i\in \{1,\dots, m\}$. 
Claim \ref{c.restriction} now implies that for every $\vartheta \in \Sigma_4 (\overline{12}\,^-)$ we have
that
$I_{\xi}^{\mathrm{t}}$ is the union 
of $m$ disjoint nontrivial intervals, proving the lemma.
\end{proof}

Consider the subset $\Sigma_4(\overline{12}_\infty^{\,-})$ of
$\Sigma_{4}$ consisting of sequences $\vartheta$ such that there is $\ell>0$ with
$\sigma^{-\ell}(\vartheta) \in \Sigma_4(\overline{12}\,^-)$. This set is dense in $\Sigma_4$.

\begin{lema}\label{l.finall} 
The spine $I_{\vartheta}^{\mathrm{t}}$ is the union of $m$ disjoint nontrivial intervals
for every $\vartheta\in \Sigma_4(\overline{12}^{-}_{\infty})$.
\end{lema}

\begin{proof}
Fix  $\vartheta\in \Sigma_4(\overline{12}^{-}_\infty)$ and let $\ell$ such that
$\sigma^{-\ell} (\vartheta)= \xi \in \Sigma_4(\overline{12}^-)$.
Observe that
$$
I_{\vartheta}^{\mathrm{t}}=f_{\vartheta_{-1}}\circ \cdots \circ f_{\vartheta_{-\ell}}(I_{\xi}^{\mathrm{t}}).
$$  
By Lemma~\ref{l.mdisjoint} the set $I_{\xi}^{\mathrm{t}}$
is a union of $m$ nontrivial disjoint intervals and since the maps $f_{i}$'s are monotone, 
we conclude that 
$I_{\vartheta}^{\mathrm{t}}$ is also  a disjoint union of $m$ nontrivial intervals, proving the lemma.
 \end{proof}
%The assertions about  
%the invariant measures in itens 3, 4 and 5 follow 
%immediately from theorem~\ref{mt.skewproduct}.   
Lemma~\ref{l.finall} and the density of $\Sigma_4(\overline{12}^{\,-}_\infty)$ in $\Sigma_4$ imply
the second part of
Proposition~\ref{p.bones}.
\end{proof}

 \section{Global attracting measures. Proof of Theorem~\ref{mt.skewproduct}}
\label{s.ergodicstepskew}
 
\subsection{Preliminaries}\label{ss.ppreliminaries}
We now introduce the main ingredients in the proof of Theorem~\ref{mt.skewproduct}.
%and Corollary \ref{mc.skewproduct2}
Recall the definition of the set
  $\mathcal{M}_{\mathbb{P}}$ of probability measures on
$\Sigma_{k}\times M$ with marginal $\mathbb{P}$ \eqref{e.marginal}.
We begin with the following (certainly well-known) proposition whose proof is included for completeness.

\begin{prop}\label{p.marginal}
The set $\mathcal{M}_\mathbb{P}$ is $F_*$-invariant and compact. 
%In particular, it contains at least
%one $F$-invariant measure.
\end{prop}

\begin{proof}
As $\mathcal{M}_\mathbb{P}$ is a subset of the compact set $\mathcal{M}_1(\Sigma_k\times M)$,
to prove its compacity 
it is enough to see that it is closed.
Let $(\mu_{n})$ be a sequence in $\mathcal{M}_{\mathbb{P}}$ 
such that $\mu_{n}\to \mu$. We need to see that $\mu\in \mathcal{M}_{\mathbb{P}}$. 
Consider the projection  on the first factor $\Pi(\vartheta,x)=\vartheta$. By the continuity 
of $\Pi$ and the definition of  $\mathcal{M}_{\mathbb{P}}$  we have that $\mathbb{P}=\Pi_*\mu_{n}$ and
$$
\mathbb{P}=\Pi_*\mu_{n}\to \Pi_*\mu,
$$
proving that $\mathbb{P}=\Pi_*\mu$ and hence $\mu \in \mathcal{M}_{\mathbb{P}}$.

To prove that
$F_{*}(\mathcal{M}_{\mathbb{P}})
\subset \mathcal{M}_{\mathbb{P}}$  note that 
$\Pi\circ F=\sigma \circ\Pi.$
Thus for $\mu\in \mathcal{M}_{\mathbb{P}}$  we have 
$$
\Pi_*\, F_{*}\mu=(\Pi\circ F)_{*}\mu=(\sigma\circ
\Pi)_*\mu =\sigma_*\, \Pi_* \mu=
\sigma_* \mathbb{P}=
\mathbb{P}.
$$
This ends the proof of  the proposition.
 \end{proof}
 
Given a compact metric space $Z$, we denote by $\mathscr{B}(Z)$ 
its Borel $\sigma$-algebra.

\begin{defi}[Disintegration of a measure]
\label{d.disintegration}
\emph{
Let $\mathbb{P}$ be a probability measure defined on $\Sigma_{k}$ and $\mu\in \mathcal{M}_{\mathbb{P}}$. 
A function
$$
  (\vartheta ,B) \in \Sigma_{k}\times \mathscr{B}(M)
 \mapsto  \mu_\vartheta (B)\in [0,1]
 $$ 
 is a \emph{disintegration of $\mu$
 with respect to $\mathbb{P}$} if 
 \begin{enumerate}
 \item
for every  $B\in \mathscr{B}(M)$ the map  
defined on $\Sigma_k$ by
$\vartheta\mapsto \mu_{\vartheta}(B)$ is $\mathscr{B}(\Sigma_{k})$-measurable,
 \item 
 for $\mathbb{P}$-a.e. $\vartheta\in \Sigma_{k}$,
 the map  
defined on $\mathscr{B}(M)$ by
 $B\mapsto \mu_{\vartheta}(B)$ is a probability measure on $(M, \mathscr{B}(M))$,
 \item
 for every $A\in \mathscr{B}(\Sigma_{k}\times M)$ 
 $$
\mu(A)= \int \mu_\vartheta (A^\vartheta)\, d\mathbb{P} (\vartheta),
$$
where 
$A^\vartheta\eqdef \{ x \in M \colon (\vartheta, x) \in  A\}$  is the
{\emph{$\vartheta$-section}} of $A$.
 \end{enumerate}
We denote the disintegration of  a measure $\mu$ above by $(\mu_\vartheta)_{\vartheta \in \Sigma_k}$}.
\end{defi}

There is the following result about existence and uniqueness of disintegrations.

\begin{prop}[Proposition 1.4.3 in \cite{Arnold}]
For every $\mu\in \mathcal{M}_{\mathbb{P}}$
its disintegration 
with respect to $\mathbb{P}$
exists and is $\mathbb{P}$-a.e. unique.
\end{prop}

We need the following definition.

 \begin{defi}[Isomorphic transformations]
 \label{d.isomorphic}
\emph{Consider pairs of probability spaces
 $(X_{1},\mathscr{B}_{1},\mu_{1})$ and $(X_{2},\mathscr{B}_{2},\mu_{2})$ and
 of measure-preserving transformations $T_{1}\colon X_{1}\to X_{1}$ and 
$T_{2}\colon X_{2}\to X_{2}$. The system $(T_{1},\mu_{1})$
 is \emph{isomorphic} to $(T_{2},\mu_{2})$ if there exist $M_{1}\in \mathscr{B}_{1}$, $M_{2}\in \mathscr{B}_{2}$
 with $\mu_{1}(M_{1})=\mu_{2}(M_{2})=1$ such that 
 \begin{enumerate}
\item 
$T_{i}(M_{i})\subset M_{i}$ for every $i=1,2$.
\item
There is an invertible measurable transformation $\phi\colon M_{1}\to M_{2}$
whose inverse is also measurable and
 satisfies $\phi_{*}\mu_{1}=\mu_{2}$ and 
$T_{2}\circ\phi(x)=\phi\circ T_{1}(x)$ for every $x\in M_{1}$.
\end{enumerate}}
\end{defi}

%
%\textcolor{magenta}{
%\begin{teo} 
%Consider a skew product map $F=F_{\mathfrak{F}}$ associated to an $\mathrm{IFS}$ $\mathfrak{F}$. 
%%as in \eqref{e.defskew}.
%Let $\lambda$ be an ergodic probability such that 
%$\lambda(S_{\mathrm{t}}^{-})=1$ . 
%Then
%\begin{enumerate}
%\item
%$F_{*|\mathcal{M}_{\lambda}}$ has a global attractor $v_{\lambda}$ and 
%the disintegration of $v_{\lambda}$
%with respect to $\lambda$ is the Dirac delta measure $\delta_{\varrho(\xi)}$.
%\item
%$(F,v_{\lambda})$ is isomorphic to 
%  $(\sigma, \lambda)$. 
%  \item
%  $\mathrm{supp}\, v_{\lambda}=\overline{\mathrm{graph} \,\,(\varrho_{|\mathrm{supp}\,\lambda}})$.
%  \item 
%   For $\lambda$-almost every sequence $\xi$ 
% and every point $x\in X$ it holds 
% $$
% \frac{1}{n}\sum_{i=0}^{n-1}\delta_{F^{i}(\omega,x)}\to v_{\lambda}.
% $$
% %In particular, the measure $v_{\lambda}$ is physical.
% \end{enumerate}
%\end{teo}}

\subsection{Proof of Theorem~\ref{mt.skewproduct}}
Recall the definitions of the set  of weakly hyperbolic sequences $S_{F}\subset \Sigma_{k}$ in \eqref{e.stmenos}, of the coding map $\rho\colon 
S_{F}\to M$  in
 \eqref{e.varrhoproj} and of the map $\hat \rho \colon S_{F}\to S_{F} \times M$ in \eqref{ultst}.

The fact that $v_\mathbb{P}=\hat \rho_\ast \mathbb{P}$ is the global attractor of $F_{*|\mathcal{M}_{\mathbb{P}}}$
 follows from the next proposition whose proof is postponed.

\begin{prop}\label{p.l.ufrj}
It holds 
$
\lim_{n\to \infty}(F^{n}_{*}\mu)_{\vartheta}=
\delta_{\rho(\vartheta)},
$
for $\mathbb{P}$-almost every $\vartheta$.
\end{prop}

%Observe that if $\omega=(\omega_i)_{i\in \mathbb{Z}}\in S_{\mathrm{t}}^-$ then
%$\omega^{-}=(\omega_{-i})_{i\ge 1} \in S_{\mathrm{t}}$ and
%$\varrho(\omega)=\pi(\omega^-)$.

To prove  item (1) in Theorem~\ref{mt.skewproduct}, take 
any continuous map $\phi \colon  \Sigma_{k}\times M \to \mathbb{R}$ and observe the following:
\[
\begin{split}
\lim_{n\to \infty}\int \phi(\vartheta,x) \, dF^{n}_{*}\mu(\vartheta,x)&
\underset{\tiny{(a)}}{=}
\lim_{n\to \infty}\int_{\Sigma_{k}}\int_{M} \phi(\vartheta,x)\,
d(F^{n}_{*}\mu)_{\vartheta}(x)\, d\mathbb{P}(\vartheta)\\
&\underset{\tiny{(b)}}{=}\int_{\Sigma_{k}} \Big( \lim_{n\to \infty}\int_{M} \phi(\vartheta,x)
d(F^{n}_{*}\mu)_{\vartheta}(x) \Big) \, d\mathbb{P}(\vartheta)\\
& \underset{\tiny{(c)}}{=}\int_{\Sigma_{k}}\int_{M} \phi(\vartheta,x)
\, d\delta_{\rho(\vartheta)}(x)\, d\mathbb{P}(\vartheta)
\underset{\tiny{(d)}}{=}\int \phi(\vartheta,x) \, d v_{\mathbb{P}}(\vartheta,x),
\end{split}
\]
where  (a) uses the definition of a disintegration,  (b) 
dominated convergence, (c) Proposition~\ref{p.l.ufrj}  and the definition of the weak$*$ limit,
and  (d)
 the definition of $v_\mathbb{P}$.
Since $\phi$ is any arbitrary continuous function, we 
conclude that $v_\mathbb{P}$ is a global attracting measure in $\mathcal{M}_\mathbb{P}$.
The continuity of $F_{*}$ immediately implies that $v_\mathbb{P}$ is the unique fixed point 
of $F_{*}$ whose marginal is $\mathbb{P}$. From the definition of $v_{\mathbb{P}}$ we have that 
$$
v_{\mathbb{P}}(A\times B)=\mathbb{P}(A\cap \rho^{-1}(B))=\int_{A} \delta_{\rho(\vartheta)}(B)\,
 d\mathbb{P}(\vartheta),
 \quad
\mbox{for every  $A\times B\subset\Sigma_{k}\times M$,}
$$
where $\mathds{1}_{\rho^{-1}(B)}$ is the characteristic function on $\rho^{-1}(B)$. This implies that
$ \delta_{\rho(\vartheta)}$ 
 is the disintegration of $v_{\mathbb{P}}$ with respect to $\mathbb{P}$,
ending the proof of item (1) in the Theorem.

\begin{proof}[Proof of Proposition~\ref{p.l.ufrj}] 
The proof of the lemma is a consequence of Lemmas~\ref{l.c.l.eraumclaim} and \ref{l.interessante}  below.

\begin{lema}
\label{l.c.l.eraumclaim}
Consider $\mu\in \mathcal{M}_\mathbb{P}$ and its disintegration $(\mu_\vartheta)_{\vartheta\in\Sigma_k}$
with respect to $\mathbb{P}$.
Then the disintegration
 of $F_{*}^{n}\mu$ with respect to $\mathbb{P}$ is given by the family of measures
 $$
 (f_{\vartheta_{-1}*}\ldots f_{\vartheta_{-n}*}\mu_{\sigma^{-n}(\vartheta)})_{\vartheta\in\Sigma_k}.
 $$
 \end{lema}

\begin{proof}
Consider any rectangle $A\times B$ in
 $\Sigma_{k}\times X$. Then, by definition of a disintegration,
 $$
F^{n}_{*}\mu(A\times B)=\mu \big( F^{-n}(A\times B) \big)=
\int \mu_{\vartheta}\big( (F^{-n}(A\times B))^{\vartheta} \big)\, d\mathbb{P}(\vartheta).
$$
 Note that if $\vartheta\in \sigma^{-n}(A)$ then
 $$
 (F^{-n}(A\times B))^{\vartheta}=(f_{\vartheta_{n-1}}\circ \cdots \circ f_{\vartheta_{0}})^{-1}(B).
 $$ 
 Otherwise,
 $(F^{-n}(A\times B))^{\vartheta}=\emptyset$. 
 Thus
\[
\begin{split}
F^{n}_{*}\mu(A\times B)&=\int_{\sigma^{-n}(A)}
\mu_{\vartheta}\big( (f_{\vartheta_{n-1}}\circ \cdots
 f_{\vartheta_{0}})^{-1}(B) \big)\, d\mathbb{P}(\vartheta)\\
&=\int (\mathds{1}_{A}\circ \sigma^{n})(\vartheta) \, \mu_{\vartheta}
\big( (f_{\vartheta_{n-1}}\circ \cdots 
\circ f_{\vartheta_{0}})^{-1}(B) \big)\, d\mathbb{P}(\vartheta).
\end{split}
\]
Defining for $n>0$ the maps
$$
g_{B,n}\colon \Sigma_{k} \to \mathbb{R}, \quad
g_{B,n}(\vartheta)\eqdef
 \mu_{\sigma^{-n}(\vartheta)}\big( (f_{\vartheta_{-1}}\circ 
 \cdots \circ f_{\vartheta_{-n}})^{-1}(B)\big)
 $$
and recalling that $\mathbb{P}$ is $\sigma$-invariant we get
 \[
 \begin{split}
 F^{n}_{*}\mu(A\times B)&=
 \int (\mathds{1}_{A}\circ \sigma^{n})(\vartheta)\,
 (g_{B,n}\circ \sigma^{n})(\vartheta)\, d\mathbb{P}(\vartheta)=\int (\mathds{1}_{A}\, g)\circ \sigma^{n}\, d\mathbb{P}(\vartheta)\\
&=\int \mathds{1}_{A}\, g\, d\mathbb{P}(\vartheta)=\int_{A}\mu_{\sigma^{-n}(\vartheta)}
((f_{\vartheta_{-1}}\circ \cdots \circ f_{\vartheta_{-n}})^{-1}(B))\,d\mathbb{P} (\vartheta)\\
&=\int_{A}f_{\vartheta_{-1}*}\ldots _{*}
f_{\vartheta_{-n}*}\mu_{\sigma^{-n}(\vartheta)}(B)\, d\mathbb{P}(\vartheta).
 \end{split}
\]
Since this identity holds for every rectangle, we get
that the family of measures 
$f_{\vartheta_{-1}*}\ldots _{*}f_{\vartheta_{-n}*}\mu_{\sigma^{-n}(\vartheta)}$
 is the disintegration of $F^{n}_{*}\mu$ with respect to $\mathbb{P}$.
\end{proof}

\begin{lema}\label{l.interessante} 
For every sequence $(\mu_{n})$
of probabilities of $\mathcal{M}_{1}(M)$  and every
 $\vartheta\in S_{F}$ it holds
$$
\lim_{n\to \infty}f_{\vartheta_{-1}*}\ldots_\ast f_{\vartheta_{-n}*}\mu_{n}=\delta_{\rho(\vartheta)}.
$$
\end{lema}

\begin{proof}
 Consider a sequence of probabilities $(\mu_n)$ and  $\vartheta\in  S_{F}$. Fix any 
$g\in C^{0}(M)$. Then  given any $\epsilon>0$ there is $\delta>0$ 
such that 
$$
|g(y)-g\circ\rho(\vartheta)|<\epsilon
\quad \mbox{for all $y\in M$ with $d(y,\rho(\vartheta))<\delta$.}
$$
Since $\vartheta \in  S_{F}$ there is $n_{0}$
such that  $d(f_{\vartheta_{-1}}\circ\cdots \circ f_{\vartheta_{-n}}(x),\rho(\vartheta))< \delta$ for every
$x\in X$ and every $n\geq n_{0}$. Therefore 
for $n\geq n_{0}$ we have
\[
\begin{split}
\left|
g\circ \rho(\vartheta)-
\int g\, df_{\vartheta_{-1}*}\ldots_\ast f_{\vartheta_{-n}*}\mu_{n}\right|&=
\left|
\int 
g\circ \rho(\vartheta)\,d\mu_{n}-
\int g\circ f_{\vartheta_{-1}}\circ\cdots\circ f_{\vartheta_{-n}}(x)\,d \mu_{n}
\right|\\
&\leq \int 
|g\circ \rho(\vartheta)- g\circ f_{\vartheta_{-1}}\circ\cdots\circ f_{\vartheta_{-n}}(x)|\, d\mu_{n}
\le \epsilon.
\end{split}
\]
This implies that
$$
\lim_{n\to \infty}\int g\, df_{\vartheta_{-1}*}\ldots_\ast f_{\vartheta_{-n}*}\mu_{n}=g\circ \rho(\vartheta)
$$
Since this holds for every continuous map $g$ the lemma follows.
\end{proof}
The proposition now follows from Lemmas~\ref{l.c.l.eraumclaim} and \ref{l.interessante}.
\end{proof}

To prove item (2) in the theorem, 
we need to construct an isomorphism between 
$(F,v_{\mathbb{P}})$ and $(\sigma,\mathbb{P})$.
Observe that the map $\hat\rho \colon S_{F}\to \mathrm{ graph}\,\rho$ defined by $
\hat\rho(\vartheta)=(\vartheta,\rho(\vartheta))$ is an invertible measurable transformation whose inverse is also measurable and given by $\hat\rho^{-1}(\vartheta,x)=\vartheta$. Note that $\sigma(S_{F})\subset S_{F}$ and $F(\mathrm{graph}\,\rho)\subset\mathrm{ graph}\,\rho$. 
It remains to see that
%$\hat\rho_{*}\mathbb{P}=v_{\mathbb{P}}$ and
$\hat\rho\circ \sigma(\vartheta)=F\circ\hat\rho(\vartheta)$ for every $\vartheta\in S_{F}$. 

%For the first assertion take any measurable rectangle $A\times B\subset \Sigma_{k}\times M$. Recalling the definition of $v_{\mathbb{P}}$ we have that 
%$$
%\hat\rho_{*}\mathbb{P}(A\times B)=\mathbb{P}( \hat\rho^{-1}(A\times B))=\mathbb{P}(A\cap \rho^{-1}(B))=v_{\mathbb{P}}(A\times B).
%$$
To get this conjugacy
recall that $f_{\vartheta_{0}}\circ\rho(\vartheta)=\rho\circ\sigma(\vartheta)$ for  every $\vartheta\in S_{F}$. Hence for  $\vartheta\in S_{F}$ we  have
$$
F\circ\hat\rho(\vartheta)=F(\vartheta,\rho(\vartheta))
=(\sigma(\vartheta),f_{\vartheta_{0}}\circ\rho(\vartheta))
=(\sigma(\vartheta),\rho\circ\sigma(\vartheta))
=\hat\rho\circ \sigma (\vartheta).
$$
Finally, since by hypothesis $\mathbb{P}(S_{F})=1$, and thus 
 $v_{\mathbb{P}}(\mathrm{graph}\,\rho)= \mathbb{P}(\hat\rho^{-1}(\mathrm{graph}\,\rho))=1$.  Therefore
$(F,v_{\mathbb{P}})$ and $(\sigma,\mathbb{P})$
are  isomorphic, ending the proof of item (2).

\medskip

We now prove item (3) in the theorem:
$\mathrm{supp}\, v_{\mathbb{P}}=\overline{\mathrm{ graph}\,(\rho_{|\mathrm{supp}\,\mathbb{P}
})}.$ 
Since $v_{\mathbb{P}}=\hat\rho_{*}\mathbb{P}$, we immediately get that
$v_{\mathbb{P}}({\mathrm{ graph}\,(\rho_{|\mathrm{supp}\,\mathbb{P}})})=1$ and hence
$$
\mathrm{supp}\, v_{\mathbb{P}}\subset \overline{\mathrm{ graph}\,(\rho_{|\mathrm{supp}\,\mathbb{P}})}.
$$
The next claim implies the inclusion ``$\supset$''  and hence the equality in item (3).
\begin{claim}
${\mathrm{graph}\,(\rho_{|\mathrm{supp}\,\mathbb{P}} )}\subset
\mathrm{supp}\, v_{\mathbb{P}}$.
 \end{claim}
 
 \begin{proof}
 Take any point $(\vartheta,x)\in {\mathrm{ graph}\,(\rho_{|\mathrm{supp}\,\mathbb{P}})}$, 
 to prove the claim it is enough to see that
 for every $\ell>0$ and every neighborhood $V$ of $x$  
 it holds
$$ 
v_\mathbb{P} \big( [\vartheta_{-\ell}\dots \vartheta_{\ell}]\times V \big)>0.
$$ 
Since $x\in V$ and $x=\rho(\vartheta)$ there is  $m\ge 0$ such that 
\begin{equation}
 \label{e.02122015}
 \rho([\vartheta_{-(\ell+m)}\dots \vartheta_{\ell}]\cap S_{F})\subset V.
 \end{equation}
 Thus by the definition of $v_{\mathbb{P}}$ 
we have 
 $$
v_\mathbb{P} \big( [\vartheta_{-\ell}\dots \vartheta_{\ell}]\times V \big)
 \geq v_{\mathbb{P}}  \big(
 [\vartheta_{-(\ell+m)}\dots \vartheta_{\ell}]\times V  \big)=
 \mathbb{P}  \big(
 [\vartheta_{-(\ell+m)}\dots \vartheta_{\ell}]\cap \rho^{-1}(V)
  \big). 
 $$ 
  By \eqref{e.02122015}
   $$
   \mathbb{P} \big( [\vartheta_{-(\ell+m)}\dots \vartheta_{\ell}]\cap \rho^{-1}(V)
   \big)= \mathbb{P}([\vartheta_{-(\ell+m)}\dots \vartheta_{\ell}]\cap S_{F}).
   $$ 
 Since that, by hypothesis,  $\vartheta \in \mathrm{supp}\, \mathbb{P} \cap [\vartheta_{-(\ell+m)}\dots \vartheta_{\ell}]$
 and  $\mathbb{P} (S_{F})=1$, 
 we get that 
 $$
 v_\mathbb{P} \big( [\vartheta_{-\ell}\dots \vartheta_{\ell}]\times V \big)
 \geq \mathbb{P}\big( [\vartheta_{-(\ell+m)}\dots \vartheta_{\ell}]\cap S_{F} \big)>0,
 $$
 which implies the claim.
 \end{proof}
 The proof of the third item of the theorem is now complete.
 \hfill \qed

\section{ Proof of Theorem~\ref{exponentialdw}}
\label{s.ergodicstepskew}
 
\subsection{Preliminaries}\label{ss.preliminaries}

Recall the definition of the canonical metric $d_0$ defined on $\Sigma_k$ in \eqref{e.do}.
In $\mathbb{R}^{m}$ we consider the sum metric $d_{1}(x,y)\eqdef \sum_{i=1}^m |x_{i}-y_{i}|$.  In this way, 
given any compact subset $M$ of $\mathbb{R}^{m}$ we can define a metric in $\Sigma_{k}\times M$ as 
follows: 
$$
d_{2}((\vartheta,x),(\xi,y))\eqdef d_{0}(\vartheta,\xi)+d_{1}(x,y).
$$
Using  the metric $d_{2}$, we define 
the space 
%of Lipschitz maps
% $\varphi\colon \Sigma_{k}\times M\to \mathbb{R}$ and 
 $\mathrm{Lip}_{1}(\Sigma_{k}\times M)$ of Lipschitz maps 
 $\varphi\colon \Sigma_{k}\times M\to \mathbb{R}$
 with Lipschitz constant one and
  the Wasserstein metric in 
the space   $\mathcal{M}_{1}(\Sigma_{k}\times M)$
of probability measures of $\Sigma_k\times M$, recall
 in \eqref{Wass}.
 
 The next proposition is a preliminary step in the proof of Theorem~\ref{exponentialdw}. Recall the splitting property in Definition~\ref{splitgene2}.

\begin{prop}\label{p.fastfast}
Let $M\subset \mathbb{R}^{m}$ be a compact subset and
 $F\colon \Sigma_{k}\times M\to \Sigma_{k}\times M$ a step skew product with
 over 
 an irreducible Markov shift $(\Sigma_k, \mathscr{F}, \mathbb{P}, \sigma)$. 
 Let  $f_{1},\dots, f_{k}$ be the fiber maps of $F$ and $P=(p_{ij})$ the transition matrix of the Markov shift.
 Assume that
 \begin{itemize}
 \item
 $F$ splits and
 \item
 there is $u$ such that $p_{uj}>0$ for every $j$. 
 \end{itemize}
 Then there are an integer $N\geq 1$ and $0<\lambda<1$ such that 
$$
\int_{\Sigma_{k}} \mathrm{diam}\,( f_{\vartheta_{-1}}\circ \dots \circ f_{\vartheta_{-n}}(M))
\, d\mathbb{P} (\vartheta) \leq m \lambda^{n}, 
\quad {\mbox{for every $n\geq N$.}}
$$

\end{prop}
\begin{proof}
 For every $n\geq 0$ and every $s$ we define the subset of $\mathbb{R}$
$$
I^{s}_{n}(\vartheta)\eqdef
  \pi_{s}(f_{\vartheta_{n-1}}\circ\dots \circ f_{\vartheta_{0}}(M))
  %=\pi_{s}(f_{\vartheta}^{n}(M))
$$ 
and let
$$
\mathbb{P}(x\in I^{s}_{n})\eqdef \mathbb{P}(\{\vartheta\in\Sigma_{k} \colon x\in I^{s}_{n}(\vartheta)\}).
$$ 

We begin with the following  lemma that is a modification of 
\cite[Lemma 6.1]{DiazMatias} about Markov measures on $\Sigma_k^+$. The lemma below is its version for $\Sigma_k$.

\begin{lema}\label{fastfast}
There is $N\geq 1$ and $0<\lambda<1$ such that for every $s$ we have 
%\margem{$N\mapsto n_0$ coherence}
$$
\mathbb{P}(x\in I^{s}_{n})\leq \lambda^{n},
\quad 
\mbox{for every $n\geq N$.}
$$ 
\end{lema}
\begin{proof}
Consider the projection 
$\Pi^{+}\colon \Sigma_{k}\to \Sigma_{k}^{+}$ given by $\Pi^{+}(\vartheta)\eqdef (\vartheta_{0}\vartheta_{1}\dots)$ for 
$\vartheta=(\vartheta_{i})_{i\in \mathbb{Z}}$. 
Note that $\mathbb{P}^{+}=(\Pi^{+})_{*}\mathbb{P}$ is a Markov measure on $\Sigma_{k}^{+}$ with the 
same transition matrix $P$ of  $\mathbb{P}$.  \cite[Lemma 6.1]{DiazMatias} claims that there is $N>0$ and $\lambda<1$ such that  
$$
\mathbb{P}^{+}(\vartheta^{+}\in \Sigma_{k}^{+}\colon x\in  \pi_{s}(f_{\vartheta^{+}_{n-1}}\circ\dots \circ f_{\vartheta^{+}_{0}}(M)))\leq \lambda^{n},
\quad 
\mbox{for every $n\geq N$}.
$$ 
Since 
$$
\mathbb{P}^{+}(\vartheta^{+}\in \Sigma_{k}^{+}\colon x\in \pi_{s}(f_{\vartheta^{+}_{n-1}}\circ\dots \circ f_{\vartheta^{+}_{0}}(M)))=\mathbb{P}(x\in I^{s}_{n})
$$ 
the lemma follows.
\end{proof}

Let {$\mathfrak{m}$}  denotes the Lebesgue measure in $\mathbb{R}$.
Note that for any sub-interval $J$ of  $\mathbb{R}$ we have  $\text{diam}\, (J)=\mathfrak{m} (J)$.
Observing that
$I_{n}^{s}(\sigma^{-n}(\vartheta))=\pi_{s}(f_{\vartheta_{-1}}\circ\dots \circ f_{\vartheta_{-n}}(M))$,
using the $\sigma$-invariance of $\mathbb{P}$,
$ \mbox{diam}\,(I_{n}^{s}(\sigma^{-n}(\vartheta)))= \mathfrak{m}(I_{n}^{s}(\vartheta)$, and
applying Fubini's theorem and Lemma \ref{fastfast} we get
\begin{equation}\label{adds}
\begin{split}
\int \mbox{diam}\, (\pi_{s}(f_{\vartheta_{-1}}\circ\dots \circ f_{\vartheta_{-n}}(M)))\, d\mathbb{P}&=
 \int \mbox{diam}\, (I_{n}^{s}(\sigma^{-n}(\vartheta)))\, d\mathbb{P}\\
 &=\int \mathfrak{m}(I_{n}^{s}(\vartheta))\, d\mathbb{P}\\
 &=
\int\mathbb{P}( x\in I_{n}^{s})\, d\mathfrak{m}\leq \lambda^{n},
\end{split}
\end{equation}
for every $n\geq N$, where $N$ and $\lambda$ are as in Lemma \ref{fastfast}. 

As 
the diameter is considered with respect to the sum metric $d_1(x,y)$, we have 
\begin{equation}\label{diam}
\mbox{diam} (X) \leq \sum_{s=1}^{m}\mbox{diam}\,(\pi_{s}(X)),
\quad \mbox{for every  compact subset $X\subset \mathbb{R}^{m}$.} \end{equation}
Hence, adding the terms in \eqref{adds} from $s=1$ to $m$ and using \eqref{diam}, it follows
$$
\int_{\Sigma_{k}} \mathrm{diam}\,( f_{\vartheta_{-1}}\circ \dots \circ f_{\vartheta_{-n}}(M))\, d\mathbb{P}\leq m
 \lambda^{n},
\quad
\mbox{for every $n\geq N$,}
$$
proving the proposition.
\end{proof}

\subsection{Proof of Theorem \ref{exponentialdw}}
 
The next lemma implies the first item of the theorem.

 \begin{lema}\label{l.PS=0}
 $\mathbb{P}(S_{F})=1$.
 \end{lema}
 
 \begin{proof}
Proposition \ref{p.fastfast}  
 implies that 
 $$
 \lim_{n\to \infty} \int_{\Sigma_{k}} \text{diam}\, (f_{\vartheta_{-1}}\circ \dots \circ f_{\vartheta_{-n}}(M))\, d\mathbb{P}(\vartheta)=0
 $$
 In particular, there is a subsequence
 $n_{k}$ such that for $\mathbb{P}$-almost every $\vartheta$ it holds
 $$
 \lim_{n\to \infty} \text{diam}\, (f_{\vartheta_{-1}}\circ \dots \circ f_{\vartheta_{-n_{k}}}(M))= 0
 $$
 Since $(f_{\vartheta_{-1}}\circ \dots \circ f_{\vartheta_{-n_{k}}}(M))_{n}$ is sequence of nested compact sets, we have 
 that $\text{diam}\, (f_{\vartheta_{-1}}\circ \dots \circ f_{\vartheta_{-n}}(M))\to 0$ for $\mathbb{P}$-a.e.
 $\vartheta$, this implies that $\mathbb{P}(S_{F})=1$.
 \end{proof}

To prove the second item of the theorem (exponential convergence to the attracting measure) we
compute 
 $d_{W}(F_{*}^n\mu,\hat{\rho}_{*}\mathbb{P})$. Recall the notation in \eqref{eq:defcompo} and observe that $f_{\vartheta}^{n}(x)$ is the second 
 coordinate of $F^{n}(\vartheta,x)$. Given $\phi\in \text{Lip}_{1}(\Sigma_{k}\times M)$
  and a probability measure $\mu$ on $\Sigma_{k}\times M$ with marginal $\mathbb{P}$  and disintegration
  $(\mu_\vartheta)_{\vartheta\in \Sigma_k}$
  we have 
  \begin{equation}\label{e.desi1}
  \begin{split}
\int \phi(\vartheta,x)\, d F_{*}^n\mu&=
\int \phi(\sigma^{n}(\vartheta),f_{\vartheta}^{n}(x))\, d\mu(\vartheta,x)\\
&=\int\int \phi(\sigma^{n}(\vartheta),
f_{\vartheta}^{n}(x))\, d\mu_{\vartheta}(x)\,d\mathbb{P}(\vartheta).
\end{split}
\end{equation}
On the other hand, from the $F$-invariance of $\hat{\rho}_{*}\mathbb{P}$,
for each $n\ge 0$,  we have that
 \begin{equation}\label{e.desi2}
 \begin{split}
\int \phi\, d \hat{\rho}_{*}\mathbb{P}&=\int \phi(\sigma^{n}(\vartheta),f_{\vartheta}^{n}(\rho(\vartheta)))\, d\mathbb{P}(\vartheta)\\
&=\int \int  \phi(\sigma^{n}(\vartheta),f_{\vartheta}^{n}(\rho(\vartheta)))\, d\mu_{\vartheta}(x)\, d\mathbb{P}(\vartheta),
\end{split}
\end{equation}
where to get the second equality we use that the integrating function only depends on $\vartheta$.
Now using that $|\phi(\vartheta,x)-\phi(\vartheta,y)|\leq d_{1}(x,y)$ (recall that $\phi\in \text{Lip}_{1}(\Sigma_{k}\times M)$)
and
 $d_{1}(f_{\vartheta}^{n}(x),f_{\vartheta}^{n}(\rho(\vartheta)))\leq 
\text{diam} (f_{\vartheta}^{n}(M))$ then  it follows from equations \eqref{e.desi1} and \eqref{e.desi2}
that
 \[
  \begin{split}
 \left|\int \phi \, d F_{*}^n\mu- \int \phi \, d \hat{\rho}_{*}\mathbb{P}\right| &\leq\int 
\text{diam} \,(f_{\vartheta}^{n}(M))\, d\mathbb{P}(\vartheta) \\
&=\int_{\Sigma_{k}} \mathrm{diam}\,(
 f_{\vartheta_{-1}}\circ \dots \circ f_{\vartheta_{-n}}(M))\, d\mathbb{P}(\vartheta),
\end{split}
\]
where the last equality follows from the $\sigma$-invariance of $\mathbb{P}$ and 
$f_{\vartheta_{-1}}\circ \dots \circ f_{\vartheta_{-n}}(M)=f_{\sigma^{-n}(\vartheta)}^{n}(M)$.
By Proposition \ref{p.fastfast} there are $N$ and $0<\lambda<1$ such that 
 $$
  \left|\int \phi(\vartheta,x)\, d F_{*}^n\mu- \int \phi(\vartheta,x)\, d \hat{\rho}_{*}\mathbb{P}\right|\leq m\lambda^{n}, \quad \mbox{for every $n\geq N$.}
 $$
 Since $\phi$ is an arbitrary function in $\text{Lip}_{1}(\Sigma_{k}\times M)$  we have that 
  $$
  d_{W}(F_{*}^n\mu,\hat{\rho}_{*}\mathbb{P})\leq m \lambda^{n}, 
 \quad 
 \mbox{for every $n\geq N$.}
 $$

To prove the last item of the theorem we need the following claim:
 
\begin{claim} There is $q<1$ such that for $\mathbb{P}$-almost every
 $\vartheta$ there is $C(\vartheta)$  such that 
$$
\mathrm{diam} \,(f_{\vartheta}^{n}(M))\leq C(\vartheta)q^{n}.
$$
\end{claim}
\begin{proof}
The proof follows using the  ideas in the proof of \cite[Theorem 2]{DiazMatias}
to obtain the synchronization of Markovian random products. Since 
$f_{\vartheta_{-1}}\circ \dots \circ f_{\vartheta_{-n}}(M)=f_{\sigma^{-n}(\vartheta)}^{n}(M)$
for every $\vartheta$, it follows from Proposition \ref{p.fastfast} and the $\sigma$-invariance of $\mathbb{P}$ that 
$$
\int \text{diam}\,(f_{\vartheta}^{n}(M))\, d\mathbb{P}(\vartheta)\leq m\lambda^{n}.
$$
In particular, taking  any $q<1$ with $\lambda<q$ and applying the 
Monotone Convergence Theorem we have that 
$$
\displaystyle\int\sum _{n=1}^{\infty}\frac{\text{diam}\,(f_{\vartheta}^{n}(M))}{q^{n}}\, d\mathbb{P}(\vartheta)=
\displaystyle\sum _{n=1}^{\infty}\frac{\int \text{diam}\,(f_{\vartheta}^{n}(M))\, d\mathbb{P}(\vartheta)}{q^{n}}<\infty.
$$ 
It follows that 
$$
\sum _{n=1}^{\infty}\frac{\text{diam}\,(f_{\vartheta}^{n}(M))}{q^{n}}< \infty,
\quad
\mbox{for 
$\mathbb{P}$-a. e. $\vartheta$.}
$$
Therefore for $\mathbb{P}$-almost every $\vartheta$ there 
is $C=C(\vartheta)$ such that 
$$
\text{diam}\,(f_{\vartheta}^{n}(M))\leq C q^{n},
$$ 
proving the claim.
\end{proof}

Since $\mathbb{P}(S_{F})=1$ then is straightforward to check that 
 for $\mathbb{P}$-almost every $\vartheta$ it holds 
$\rho(\sigma^{n}(\vartheta))=f_{\vartheta}^{n}(\rho(\vartheta))$. Hence
it follows from the previous claim that for $\mathbb{P}$-almost every $\vartheta$ there is $C(\vartheta)$ 
such that for every $x$ we have 
that 
$$
 d_2(F^{n}(\vartheta,x),\hat\rho(\sigma^{n}(\vartheta))=
 d_{1}(f_{\vartheta}^{n}(x),f_{\vartheta}^{n}(\rho(\vartheta)))\leq \text{diam}\,(f_{\vartheta}^{n}(M))\leq C(\vartheta) q^{n}, 
$$ 
proving item ($4$).
 \hfill \qed

\bibliographystyle{acm}
\bibliography{references}

\end{document}